\documentclass[a4paper,twoside,12pt]{article}
\usepackage{xeCJK}
\usepackage{titlesec}
\usepackage{titletoc}
\usepackage{amsmath, amsfonts}
\usepackage{pdfpages}
\usepackage{enumerate}
\usepackage{geometry}
\usepackage{color}
\usepackage[all,pdf]{xy}
\usepackage{verbatim}
\usepackage{amssymb}
\usepackage{amsthm}
\usepackage{enumitem}
\usepackage{hyperref}
\usepackage[T1]{fontenc}
\usepackage{fancyhdr} 
\usepackage[none]{hyphenat}

\newcommand{\xiaoerhao}{\fontsize{18pt}{\baselineskip}\selectfont}

\newcommand{\xiaosanhao}{\fontsize{15pt}{\baselineskip}\selectfont}
\newcommand{\sihao}{\fontsize{14pt}{\baselineskip}\selectfont}
\newcommand{\xiaosihao}{\fontsize{12pt}{\baselineskip}\selectfont}

\titleformat{\section}{\centering\xiaosanhao\bfseries}{\thesection}{1em}{}
\titleformat{\subsection}{\flushleft\sihao\bfseries}{\thesubsection}{1em}{}
\titleformat{\subsubsection}{\flushleft\xiaosihao\bfseries}{\thesubsubsection}{1em}{}

\newtheorem{theorem}{Theorem}[section]

\newtheorem{lemma}[theorem]{Lemma}

\newtheorem{definition}[theorem]{Definition}

\newtheorem{example}[theorem]{Example}
\newtheorem*{myremark}{Remark}

\numberwithin{equation}{section}

\title{Gauss Genus Theory in Characteristic 2}
\author{Qiyu Zhang}
\date{}

\begin{document}

\newpage
%% ================================================

\pagenumbering{arabic}
\renewcommand{\thepage}{\arabic{page}}

\maketitle

\begin{abstract}
	We extend Gauss composition and Gauss genus theory over $\mathbf{Z}$ to $\mathbb{F}_{2^n}[T]$, a polynomial ring over a finite field $\mathbb{F}_{2^n}$ of characteristic 2. We find new invariants of binary quadratic forms over $\mathbb{F}_{2^n}[T]$ by using Arf invariant and introduce new definitions of proper equivalence and direct composition, and prove that the direct composition makes the set of proper equivalence classes of binary quadratic forms with the same invariants into a finite Abelian group, which is isomorphic to a Picard group of a corresponding extension ring of $\mathbb{F}_{2^n}[T]$.
	Building on this, we develop genus theory in characteristic 2 and prove that the kernel of the generalized Gauss's map is the subgroup of all squares in the class group.
	
\end{abstract}

%% ================================================

\section{Introduction} 
\subsection{Background}
\qquad
Gauss composition and Gauss genus theory are fundamental results in the theory of binary quadratic forms over $\mathbf{Z}$. We say that two quadratic form $f_i(x,y)=a_ix^2+b_ixy+c_iy^2,i=1,2$ are proper equivalence, if there exists $\gamma \in Sl_2(\mathbf{Z})$ such that $f_1(x,y)=f_2((x,y)\gamma)$. Suppose $D \in \mathbf{Z}$ is not a perfect square and $D \equiv 0,1 \mod{4}$. Let \[Cl^+(D):=\{[f]\; | \; f \text{ is a primitive binary quadratic form of discriminant } D\},\]where $[f]$ is the proper equivalence class of $f$.

Legendre $^{\cite[ pp.361-379]{Legendre}}$ first considered the composition of quadratic forms over $\mathbf{Z}$. More precisely, let $f(x,y)$ and $g(z,w)$ be quadratic forms of discriminant $D$, then a form $F(x,y)$ of the same discriminant is their composition provided that 
\begin{equation}\label{0.1}\begin{gathered}
f(x,y)g(z,w)=F(B_1(x,y;z,w),B_2(x,y;z,w)),\\
\text{where } B_i(x,y;z,w)=p_ixz+q_ixw+r_iyz+s_iyw,\quad i=1,2.\end{gathered}\end{equation}
Legendre showed that any two forms of the same discriminant can be composed, but two forms can be composed in many different ways, and the resulting forms need not be properly equivalent. 

Given the above composition data, Gauss $^{\cite[ pp.230-259]{Gauss}}$ proved that
\begin{equation}\label{0.2}p_1q_2-p_2q_1=\pm f(1,0), \; p_1r_2-p_2r_1=\pm g(1,0),\end{equation}
and then he defined the composition to be a direct composition provided that both of the signs in the equations above are $+$. Gauss proved that for a fixed discriminant, direct composition makes $Cl^+(D)$ into a finite Abelian group. Moreover, Dirichlet $^{\cite[pp.611-627]{Dirichlet}}$ proved that $Cl^+(D)$ is isomorphic to $ \mathbb{CL}^+(\mathbf{Z} \oplus \mathbf{Z} \frac{D+\sqrt{D}}{2})$, the narrow ideal class group of $\mathbf{Z} \oplus \mathbf{Z} \frac{D+\sqrt{D}}{2}$.

Based on Gauss composition, Gauss $^{\cite[pp.270-277]{Gauss}}$ established the genus theory. Gauss defined a map \begin{equation}\label{0.3}
	\omega : Cl^+(D) \to (\mathbf{Z}/D\mathbf{Z})^*/H,
\end{equation} which maps $[f]$ to the elements in $(\mathbf{Z}/D\mathbf{Z})^*$ that can be represented by $[f]$. The kernel of $\omega$ is $Cl^+(D)^2$.

\subsection{Main Results}
\qquad Little research has been conducted on Gauss composition and Gauss genus theory of binary quadratic forms on polynomial ring of a field of characteristic 2. The first difficulty is that the invariants of quadratic forms in characteristic 2 differ from those in other characteristics since quadratic forms in characteristic 2 cannot correspond to a symmetric matrix. To solve this problem, Arf $^{\cite[ p.153]{Arf}}$ provided a new invariant. Suppose $k$ is a field of characteristic 2, and
$f=ax^2+bxy+cy^2$ is a quadratic form over $k$, Arf found that \[\Delta(f) \equiv \frac{ac}{b^2} \mod{\varphi(k)}, \text{ where } \varphi(x):=x^2+x\]
is an invariant under linear transformations. As we will detail in section 2.1, Arf's theory carries over to polynomial rings after minor modifications. 

The main challenge of extend Gauss composition to characteristic 2 is that the definition of proper equivalence and direct composition of quadratic forms over $\mathbf{Z}$ cannot be directly applied to characteristic 2, and one example in section 2.2 shows that two forms can be composed in many different ways (similarly to the case over $\mathbf{Z}$), so equation $(\ref{0.1})$ can not induce an operation on the set of equivalence class. Thus, it is necessary to find the analogue of direct composition. Also, we must consider whether it is necessary to define proper equivalence of forms and narrow equivalence of fractional ideal.

The remainder of this paper is organized as follows. In section 2, Gauss composition is extended to $\mathbb{F}_{2^n}[T]$, a polynomial ring of a finite field of characteristic 2. In section 2.1, the invariant and proper equivalence of quadratic forms over $\mathbb{F}_{2^n}[T]$ are discussed, then $C(b,\Delta)$ is defined as the set of proper equivalence class of forms with invariants $b$ and $\Delta$. In section 2.2, direct composition of forms over $\mathbb{F}_{2^n}[T]$ is defined, which induces an operation on $C(b,\Delta)$. In section 2.3, $C(b,\Delta)$ is shown to be a finite Abelian group isomorphic to a Picard group of corresponding extension ring of $\mathbb{F}_{2^n}[T]$.
In section 3, Gauss genus theory is extended to quadratic forms over $\mathbb{F}_{2^n}[T]$. In section 3.1, a map $\omega: C(b,\Delta) \to (R/(b^2))^*/H$ is defined as a generalization of map $(\ref{0.3})$. In section 3.2, $\ker(\omega)$ is shown to be $C(b,\Delta)^2$.

Most of the proofs in this paper are inspired by the theory of quadratic forms over $\mathbf{Z}$, and some definition and many details have been adjusted to accommodate the difference between $\mathbf{Z}$ and $\mathbb{F}_{2^n}[T]$.

\section{Composition of Binary Quadratic Forms}
\subsection{Proper Equivalence of Binary Quadratic Forms}
\qquad
Let $\mathbb{F}_{2^n}$ be a field of $2^n$ elements. Let $R=\mathbb{F}_{2^n}[T]$ be a polynomial ring. Let $K=Frac(R)$ be the fractional field of $R$.

Denote binary quadratic form $ax^2+bxy+cy^2$ by $(a,b,c)$, $a,b,c \in R $. 

Let $Sl_2(R)=\{\begin{pmatrix}p & q \\r & s\end{pmatrix}|ps+qr=1,p,q,r,s\in R\}$. 
Suppose \[f(x,y)=ax^2+bxy+cy^2,\; \gamma = \begin{pmatrix}p & q \\r & s\end{pmatrix} \in Sl_2(R).\]
Then $F(x,y) := f((x,y)\gamma)=f(px+ry,qx+sy)=Ax^2+Bxy+Cy^2$, where
\begin{equation}\label{1}
	\begin{cases}
		A=ap^2+bpq+cq^2,\\
		B=b(ps+qr)=b,\\
		C=ar^2+brs+cs^2,
	\end{cases}
\end{equation}
denoted by $\gamma(a,b,c)=(A,B,C)$ or $\gamma f = F$, and call (a,b,c) and (A,B,C) are equivalent. By definition, $\gamma_1(\gamma_2(a,b,c))=(\gamma_1\gamma_2)(a,b,c)$. 

If $\gamma \in Sl_2(R)$, and $(a,b,c)$ is a primitive form( i.e. $a,b,c$ are coprime), then the common divisors of $A,B,C$ also divide $a,b,c$ since $(a,b,c)=\gamma^{-1}(A,B,C)$, then   $(A,B,C)=\gamma (a,b,c)$ \linebreak is a primitive form.

Let $\varphi (t)=t^2+bt$. From Arf invariant we can find \begin{equation}\label{Arf}
	AC=ac+\varphi(apr+bqr+cqs).
\end{equation} So if $(a_1,b_1,c_1)$ and $(a_2,b_2,c_2)$ are equivalent, we have \[b_1=b_2, \: a_1c_1 \equiv a_2c_2 \mod \varphi(R).\]

\begin{lemma}\label{1.1}
	$\forall b \in R$, there exists a subgroup$R_b$ of $(R,+)$, such that $\forall k \in R$, exactly one of $k$ and $k+b$ belongs to $R_b$.
\end{lemma}
\begin{proof} Choose a subgroup G of $(\mathbb{F}_{2^n},+)$ such that $\# G=2^{n-1}$ and $ 1 \notin G$. Let \[R_b=\{ qb+r | q,r \in R, deg(r)< deg(b), q(0) \in G\}.\qedhere\]
\end{proof}

Fix $G$ and $R_b$ from now on.

\begin{definition}\label{1.2}
	Call primitive quadratic form $(a,b,c)$ is proper equivalent to $(A,B,C)$ ( denote by $(a,b,c) \approx (A,B,C)$), if there exists $\gamma=\begin{pmatrix}p & q \\r & s\end{pmatrix} \in Sl_2(R)$, such that \[\gamma(a,b,c)=(A,B,C), \text{ and } apr+bqr+cqs \in R_b.\] Denote $\{\gamma=\begin{pmatrix}p & q \\r & s\end{pmatrix} \in Sl_2(R)|apr+bqr+cqs \in R_b\}$ by $S_{(a,b,c)}$.
\end{definition} 

\begin{myremark}
	This definition is inspired by equation $(\ref{Arf})$, lemma $\ref{1.19}$ and equation $(\ref{13})$. 
\end{myremark}

Claim that proper equivalence is an equivalence relation. Let $\gamma = I_2 =\begin{pmatrix}1 & 0 \\0 & 1\end{pmatrix}$, then $\gamma(a,b,c)=(a,b,c)$, and $0 \in R_b$, so $(a,b,c) \approx (a,b,c)$. If \[ \begin{pmatrix}p & q \\r & s\end{pmatrix}(a,b,c)=(A,B,C),\: apr+bqr+cqs \in R_b,\] then \[\begin{pmatrix}s & q \\r & p\end{pmatrix}(A,B,C)=(a,b,c),\] and 
\begin{equation}\notag
	\begin{split}
		Asr+Bqr+Cqp &=(ap^2+bpq+cq^2)sr+bqr+(ar^2+brs+cs^2)qp \\&=apr(ps+qr)+bqr+cqs(qr+ps)\\&=apr+bqr+cqs \in R_b.
	\end{split}
\end{equation}
So $(a,b,c) \approx (A,B,C)$ implies $(A.B,C) \approx (a,b,c)$. Finally, if \[\begin{pmatrix}p_1 & q_1 \\r_1 & s_1\end{pmatrix}(a_1,b,c_1)=(a_2,b,c_2),\begin{pmatrix}p_2 & q_2 \\r_2 & s_2\end{pmatrix}(a_2,b,c_2)=(a_3,b,c_3),\]\[a_ip_ir_i+bq_ir_i+c_iq_is_i \in R_b, i=1,2,\] then\[ (\begin{pmatrix}p_2 & q_2 \\r_2 & s_2\end{pmatrix}\begin{pmatrix}p_1 & q_1 \\r_1 & s_1\end{pmatrix})(a_1,b,c_1)=(a_3,b,c_3),\]and 
\begin{equation}\label{2}
	\begin{split}
		&a_1(p_2p_1+q_2r_1)(r_2p_1+s_2r_1)+b(r_2p_1+s_2r_1)(p_2q_1+q_2s_1)+c_1(p_2q_1+q_2s_1)(r_2q_1+s_2s_1)\\=&(a_1p_1^2+bp_1q_1+c_1q_1^2)p_2r_2+bq_2r_2(p_1s_1+q_1r_1)+(a_1r_1^2+br_1s_1+c_1s_1^2)q_2s_2\\&+(a_1p_1r_1+bq_1r_1+c_1q_1s_1)(p_2s_2+q_2r_2)\\=&(a_2p_2r_2+bq_2r_2+c_2q_2s_2)+(a_1p_1r_1+bq_1r_1+c_1q_1s_1) \in R_b.
	\end{split}
\end{equation}
So $(a_1,b,c_1) \approx (a_2,b,c_2),(a_2,b,c_2) \approx (a_3,b,c_3)$ implies $(a_1,b,c_1) \approx (a_3,b,c_3)$. Then proper equivalence is an equivalence relation.

Moreover, suppose $\begin{pmatrix}p_1 & q_1 \\r_1 & s_1\end{pmatrix}(a_1,b,c_1)=(a_2,b,c_2)$, $\begin{pmatrix}p_2 & q_2 \\r_2 & s_2\end{pmatrix}(a_2,b,c_2)=(a_3,b,c_3)$, and $a_ip_ir_i+bq_ir_i+c_iq_is_i \notin R_b, i=1,2$. By equation $(\ref{2})$ and  \[(a_2p_2r_2+bq_2r_2+c_2q_2s_2)+(a_1p_1r_1+bq_1r_1+c_1q_1s_1) \in R_b,\] we have $\begin{pmatrix}p_2 & q_2 \\r_2 & s_2\end{pmatrix}\begin{pmatrix}p_1 & q_1 \\r_1 & s_1\end{pmatrix} \in S_{(a_1,b,c_1)}$. So the next lamma follows:

\begin{lemma}\label{1.3}
	Suppose $(a,b,c)$ is a primitive quadratic form and $\gamma_1,\gamma_2 \in Sl_2(R)$. Then
	\begin{enumerate}[label=(\arabic*)]
		\item If $\gamma_1 \in S_{(a,b,c)}$, then $\gamma_1^{-1}\in S_{\gamma_1(a,b,c)}$.
		\item $\gamma_2\gamma_1 \in S_{(a,b,c)}$ if and only if $\gamma_1 \in S_{(a,b,c)}$ and $\gamma_2 \in S_{\gamma_1(a,b,c)}$, or $\gamma_1 \notin S_{(a,b,c)}$ and \\ $\gamma_2 \notin S_{\gamma_1(a,b,c)}$.
	\end{enumerate} 
\end{lemma}

Since proper equivalence is an equivalence relation, we can denote the proper equivalence class of $(a,b,c)$ by $[(a,b,c)]$.

\begin{definition}\label{1.4}
	Suppose $b \in R$, $b \ne 0$, $\Delta \in R$, $\Delta \not \equiv 0 \mod{\varphi(R)}$. Define \[C(b,\Delta):= \{[(a,b,c)]|(a,b,c) \text{ is a primitive quadratic form, }a,c \in R,\; ac \equiv \Delta \mod{\varphi(R)} \}.\]
\end{definition} 

\subsection{Composition of Binary Quadratic Forms}
\qquad

Fix $b \in R$, $b \ne 0$, $\Delta \in R$, $\Delta \not \equiv 0 \mod{\varphi(R)}$.

For $[(a,b,c)] \in C(b,\Delta)$, there exists $k \in R $ such that $ac=\Delta + \varphi(k)=\Delta + \varphi(k+b)$. Exactly one of $k$ and $k+b$ belongs to $R_b$. Take $k \in R_b$ for convention.

Legendre gives the definition of composition of quadratic forms: 
\begin{definition}\label{1.5}
	The quadratic form $(a_3,b,c_3)$ is called the composition of $(a_1,b,c_1)$ and $(a_2,b,c_2)$, if there exists $p_i,q_i,r_i,s_i \in R$, such that \begin{equation}\label{3}
		(a_1x^2+bxy+c_1y^2)(a_2z^2+bzw+c_2w^2)=a_3L_1^2+bL_1L_2+c_3L_2^2, 
	\end{equation}
	where $L_i=p_ixz+q_ixw+r_iyz+s_iyw,i=1,2$.
\end{definition}

Bhargava$^{\cite[ pp.219-220]{Bhargava}}$ use $2\times2\times2$ cubes to give a description of the composition of quadratic forms over $\mathbf{Z}$. Similarly, $2\times2\times2$ cubes over $R$ can give a description of the composition of quadratic forms over $R$ can be described. Suppose $a,b,c,d,e,f,g,h \in R$. Let $\mathcal{A}=$ be the $2\times2\times2$ cube
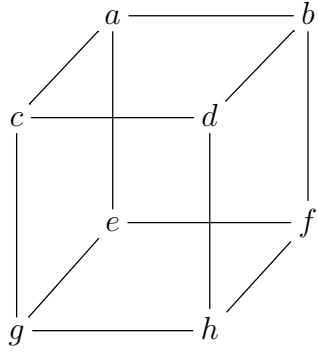
\[\xymatrix{&
	a\ar@{-}[rr]\ar@{-}[ld]\ar@{-}[dd]& &
	b\ar@{-}[ld]\ar@{-}[dd]\\
	c\ar@{-}[rr]\ar@{-}[dd]& &
	d\ar@{-}[dd]&\\&
	e\ar@{-}[rr]\ar@{-}[ld]& &
	f\ar@{-}[ld]\\
	g\ar@{-}[rr]& &
	h&}\]
$\mathcal{A}$ can be partitioned into the $2\times2$ matrices\[M_1=\begin{pmatrix}a&b\\c&d\end{pmatrix},N_1=\begin{pmatrix}e&f\\g&h\end{pmatrix},\]or into \[M_2=\begin{pmatrix}a&c\\e&g\end{pmatrix},N_2=\begin{pmatrix}b&d\\f&h\end{pmatrix},\]or\[M_3=\begin{pmatrix}c&d\\g&h\end{pmatrix},N_3=\begin{pmatrix}a&b\\e&f\end{pmatrix}.\]
Then the cube gives three quadratic forms, \[f_i^{\mathcal{A}}(x,y)=det(M_ix+N_iy)=a_ix^2+b_ixy+c_iy^2,i=1,2,3.\] More precisely, 
\begin{equation}\label{4}\begin{split}
	&f_1^{\mathcal{A}}(x,y)=(ad+bc)x^2+(ah+bg+de+cf)xy+(eh+fg)y^2,\\
	&f_2^{\mathcal{A}}(x,y)=(ag+ce)x^2+(ah+bg+de+cf)xy+(bh+df)y^2,\\
	&f_3^{\mathcal{A}}(x,y)=(ch+dg)x^2+(ah+bg+de+cf)xy+(af+be)y^2.
\end{split}\end{equation}
Suppose $\gamma \in Sl_2(r)$. Let $^i\gamma(i=1,2,3)$ be the act on the cube $\mathcal{A}$ that replace $(M_i,N_i)$ by $(pM_i+qN_i,rM_i+sN_i)$. Then \begin{equation}\label{5}
	f_j^{^i\gamma \mathcal{A}}=\begin{cases}\gamma f_j^{\mathcal{A}},&i=j;\\f_j^{\mathcal{A}},&i \ne j.
	\end{cases}
\end{equation}

By the above definition the composition can be represented by a $2\times2\times2$ cube:

\begin{lemma}\label{1.6}
    Suppose $a_i,b,c_i(i=1,2,3),p_j,q_j,r_j,s_j(j=1,2) \in R, a_1 \ne 0, b \ne 0$. Let $\mathcal{A}$ be the cube 
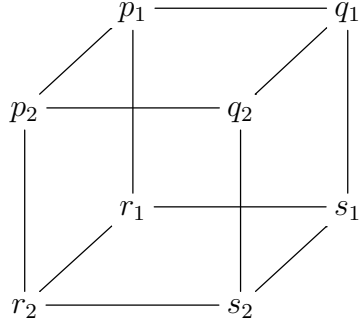
\[\xymatrix{&
		p_1\ar@{-}[rr]\ar@{-}[ld]\ar@{-}[dd]& &
		q_1\ar@{-}[ld]\ar@{-}[dd]\\
		p_2\ar@{-}[rr]\ar@{-}[dd]& &
		q_2\ar@{-}[dd]&\\&
		r_1\ar@{-}[rr]\ar@{-}[ld]& &
		s_1\ar@{-}[ld]\\
		r_2\ar@{-}[rr]& &
		s_2&}\]
	Then $a_i,b,c_i(i=1,2,3),p_j,q_j,r_j,s_j(j=1,2)$ satisfy the equation $(\ref{3})$ if and only if
	\[f_i^{\mathcal{A}}=(a_i,b,c_i),i=1,2,3.\]
\end{lemma}
\begin{proof} Compute the coefficients of $x^2z^2,x^2w^2,y^2z^2,y^2w^2,x^2zw,xyz^2,xyw^2,y^2zw,xyzw$ in equation $(\ref{3})$, then the equation $(\ref{3})$ is equivalent to 
\begin{equation}\label{6} \left\{\begin{aligned}
	&a_1a_2=a_3p_1^2+bp_1p_2+c_3p_2^2,\\
	&a_1c_2=a_3q_1^2+bq_1q_2+c_3q_2^2,\\
	&c_1a_2=a_3r_1^2+br_1r_2+c_3r_2^2,\\
	&c_1c_2=a_3s_1^2+bs_1s_2+c_3s_2^2,\\
	&a_1b=(p_1q_2+p_2q_1)b,\\
	&a_2b=(p_1r_2+p_2r_1)b,\\
	&c_2b=(q_1s_2+q_2s_1)b,\\
	&c_1b=(r_1s_2+r_2s_1)b,\\
	&b^2=b(p_1s_2+p_2s_1+q_1r_2+q_2r_1).
\end{aligned}\right.\end{equation}
Since $b \ne 0$, the last five equations are equivalent to \begin{gather}
	a_1=(p_1q_2+p_2q_1),\label{7}\\ 
	a_2=(p_1r_2+p_2r_1),\label{8}\\
	c_2=(q_1s_2+q_2s_1),\label{9}\\
	c_1=(r_1s_2+r_2s_1),\label{10}\\
	b=p_1s_2+p_2s_1+q_1r_2+q_2r_1,\label{11}
\end{gather}
which are equivalent to $f_i^{\mathcal{A}}=(a_i,b,c_i),i=1,2$.

Now suppose equation $(\ref{7})-(\ref{11})$ holds. From the first two equations of $(\ref{6})$we have\begin{equation}
	\left\{\begin{aligned}
		&p_1^2a_3+p_2^2c_3=a_1a_2+bp_1p_2,\\
		&q_1^2a_3+q_2^2c_3=a_1c_2+bq_1q_2.
	\end{aligned}\right.
\end{equation}
By \[p_1^2q_2^2-p_2^2q_1^2=(p_1q_2+p_2q_1)^2=a_1^2 \ne 0\] and Cramer's rule, for given $a_i,b,c_i(i=1,2),p_j,q_j,r_j,s_j(j=1,2)$, there is at most one pair $(a_3,c_3)$ satisfy $(\ref{6})$. On the other hand, \[\begin{aligned}
	&(p_2s_2+q_2r_2)p_1^2+bp_1p_2+(p_1s_1+q_1r_1)p_2^2\\
	=&(p_2s_2+q_2r_2)p_1^2+(p_1s_2+p_2s_1+q_1r_2+q_2r_1)p_1p_2+(p_1s_1+q_1r_1)p_2^2\\
	=&(p_1q_2+p_2q_1)(p_1r_2+p_2r_1)\\=&a_1a_2,
\end{aligned}\]
and similarly we have
\[\begin{aligned}
	&(p_2s_2+q_2r_2)q_1^2+bq_1q_2+(p_1s_1+q_1r_1)q_2^2=a_1c_2,\\
	&(p_2s_2+q_2r_2)r_1^2+br_1r_2+(p_1s_1+q_1r_1)r_2^2=c_1a_2,\\
	&(p_2s_2+q_2r_2)s_1^2+bs_1s_2+(p_1s_1+q_1r_1)s_2^2=c_1c_2.
\end{aligned}\]
That is, if $a_3=p_2s_2+q_2r_2,c_3=p_1s_1+q_1r_1$, equation $(\ref{6})$ will hold. So the first four equation of $(\ref{6})$ is equivalent to $a_3=p_2s_2+q_2r_2,c_3=p_1s_1+q_1r_1$ and $f_3^{\mathcal{A}}=(a_3,b,c_3)$.\end{proof}

\begin{example}\label{eg0}
	Let 
	\[\mathcal{A}= \begin{gathered}\xymatrix{&
			1\ar@{-}[rr]\ar@{-}[ld]\ar@{-}[dd]& &
			0\ar@{-}[ld]\ar@{-}[dd]\\
			0\ar@{-}[rr]\ar@{-}[dd]& &
			T\ar@{-}[dd]&\\&
			0\ar@{-}[rr]\ar@{-}[ld]& &
			T\ar@{-}[ld]\\
			T\ar@{-}[rr]& &
			1&}\end{gathered},\quad 
	\mathcal{A'}= \begin{gathered}\xymatrix{&
			T\ar@{-}[rr]\ar@{-}[ld]\ar@{-}[dd]& &
			1\ar@{-}[ld]\ar@{-}[dd]\\
			0\ar@{-}[rr]\ar@{-}[dd]& &
			1\ar@{-}[dd]&\\&
			0\ar@{-}[rr]\ar@{-}[ld]& &
			T^2\ar@{-}[ld]\\
			1\ar@{-}[rr]& &
			0&}\end{gathered}.\]
	Then \[f_1^{\mathcal{A}}=f_1^{\mathcal{A'}}=(T,1,T^2),\; f_2^{\mathcal{A}}=f_2^{\mathcal{A'}}=(T,1,T^2),\; f_3^{\mathcal{A}}=(T^2,1,T),\; f_3^{\mathcal{A'}}=(1,1,T^3).\]
	For all $p,q \in R$ we have $\deg(p^2+pq+T^3q^2) \ne 1$ so $(1,1,T^3)$ is not equivalent to $(T,1,T^2)$. That is, equation $(\ref{3})$ can not induce an operation on the set of equivalence (or proper equivalence) class and we need to add extra restrictions.
\end{example}

The next lemma is the analogue of equation $(\ref{0.2})$ in Gauss composition over $\mathbf{Z}$.

\begin{lemma}\label{1.7}
	If $a_1c_1 = a_2c_2 + \varphi(k)$, and $p_i,q_i,r_i,s_i(i=1,2)$ satisfy equation $(\ref{7})-(\ref{11})$, then $p_1s_2+p_2s_1=k,q_1r_2+q_2r_1=k+b$ or $p_1s_2+p_2s_1=k+b,q_1r_2+q_2r_1=k$.
\end{lemma}
\begin{proof}By $p_1s_2+p_2s_1+q_1r_2+q_2r_1=b \ne 0$, $p_1s_2+p_2s_1$ and $q_1r_2+q_2r_1$ are not both zero. Without loss of generality, let $d:=p_1s_2+p_2s_1 \ne 0$. From equation $(\ref{7})$ and $(\ref{9})$ we have 
\begin{equation}\notag
	\left\{\begin{aligned}
		&p_2q_1+p_1q_2=a_1,\\
		&s_2q_1+s_1q_2=c_2,
	\end{aligned}\right.
\end{equation} 
So by using Cramer's rule we find \[q_1=\frac{a_1s_1+c_2p_1}{d},q_2=\frac{p_2c_2+a_1s_2}{d}.\]Similarly, from equation $(\ref{8})$ and $(\ref{10})$ we can find \[r_1=\frac{a_2s_1+c_1p_1}{d},r_2=\frac{p_2c_1+a_2s_2}{d}\] by using Cramer's rule. Then,
\[\begin{aligned}
	b&=p_1s_2+p_2s_1+q_1r_2+q_2r_1\\
	&=(p_1s_2+p_2s_1)+\frac{a_1s_1+c_2p_1}{d}\cdot\frac{p_2c_1+a_2s_2}{d}+\frac{p_2c_2+a_1s_2}{d}\cdot\frac{a_2s_1+c_1p_1}{d}\\
	&=d+\frac{(a_1c_1+a_2c_2)(p_1s_2+p_2s_1)}{d^2}\\
	&=d+\frac{a_1c_1+a_2c_2}{d}\\
	&=d+\frac{k^2+bk}{d}.
\end{aligned}\]
If $k=0$, we have proved that $p_1s_2+p_2s_1=d=b$ and $q_1r_2+q_2r_1=b-d=0$. If $k \ne 0$, from $b=d+\frac{k^2+bk}{d}$ we find $d=k$ or $d=k+b$. So $p_1s_2+p_2s_1=k,q_1r_2+q_2r_1=k+b$ or $p_1s_2+p_2s_1=k+b,q_1r_2+q_2r_1=k$.\end{proof}
\begin{myremark}
	We can verify $\varphi(p_1s_2+p_2s_1)=a_1c_1+a_2c_2$ by direct computation and obtain a shorter proof, but the proof above can show the idea of how to find a Bhargava cube $\mathcal{A}$ such that $f_1^{\mathcal{A}}$ and $f_2^{\mathcal{A}}$ are given forms. If the condition $a_1c_1=a_2c_2+\varphi(k)$ is removed, we can still prove that 
	\[\varphi(p_1s_2+p_2s_1)=a_1c_1+a_2c_2,\]
	that is, equation $(\ref{3})$ implies that $a_1c_1 \equiv a_2c_2 \mod{\varphi(R)}$.
\end{myremark}

Analogous to the definition of direct composition of quadratic forms over $Z$ given by Gauss, we can define direct composition of quadratic forms over $R=\mathbb{F}_{2^n}[T]$ such that direct composition makes the set of classes of quadratic forms over $R$ into an Abelian group. From \[p_1s_2+p_2s_1+q_1r_2+q_2r_1=b\] we can find exactly one of $p_1s_2+p_2s_1$ and $q_1r_2+q_2r_1$ belongs to $R_b$, furthermore, under the condition of lemma $\ref{1.7}$, $p_1s_2+p_2s_1$ and $q_1r_2+q_2r_1$ can be determined if we know which one of them is in $R_b$. Symmetrically, there is exactly one of $p_2s_1+q_2r_1$ and $p_1s_2+q_1r_2$ in $R_b$. So we can give the next definition:

\begin{definition}\label{1.8}
	The quadratic form $(a_3,b,c_3)$ is called a direct composition of $(a_1,b,c_1)$ and $(a_2,b,c_2)$, if there exists $p_i,q_i,r_i,s_i,i=1,2$, such that equation $(\ref{3})$ holds and \[q_1r_2+q_2r_1 \in R_b, p_2s_1+q_2r_1 \in R_b.\] Call the cube in lemma $\ref{1.6}$ is direct if $q_1r_2+q_2r_1 \in R_b,p_2s_1+q_2r_1 \in R_b$.
\end{definition}

We shall next prove that direct composition makes the set of classes of quadratic forms over $R$ into an Abelian group. The first step is to show that for all $[(A_1,b,C_1)],[(A_2,b,C_2)]$ in $ C(b,\Delta)$, we can choose $(a_1,b,c_1) \in [(A_1,b,C_1)]$ and $(a_2,b,c_2) \in [(A_2,b,C_2)]$ such that we can find a direct composition of them easily.

\begin{lemma}\label{1.9}
	$\forall u \in R$, if $b \nmid u$, then $\exists v \in R$, such that $uv=qb+r$, where $\deg r < \deg b$, and $q(0)=1$.
\end{lemma}
\begin{proof}There exists $e_1,e_2 \in R$, such that $ue_1+be_2=(u,b)$, and $\deg(u,b) < \deg b$ since $b \nmid u$. If $e_2(0) \ne 0$, let $v = e_2(0)^{-1}e_1$, then \[uv=ue_2(0)^{-1}e_1=e_2(0)^{-1}(e_2b+(u,b)),\quad\deg e_2(0)^{-1}(u,b)<\deg b \quad\text{and}\quad e_2(0)^{-1}e_2(0)=1. \]If $e_2(0) = 0$, let $v=\frac{e_1b}{(u,b)}$, then \[u\frac{e_1b}{(u,b)}=(\frac{e_2b}{(u,b)}+1)b,\quad(\frac{e_2b}{(u,b)}+1)(0)=1.\qedhere\]\end{proof}

The first paragraph of the proof of next lemma is from lemma 8.2.1 of Fang$^{\cite[ p.177]{Fang}}$.

\begin{lemma}\label{1.10}
	Suppose $(a,b,c)$ is a primitive quadratic form. $\forall m \in R$, there exists a primitive quadratic form $(A,b,C)$, such that $(A,b,C) \approx (a,b,c)$ and $(A,m)=1$.
\end{lemma}
\begin{proof}We first claim that $\exists p,q \in R $, such that $(ap^2+bpq+cq^2,mb)=1$ and $(p,q)=1$. Let $p$ be the product of all primes that divides $mb$ and $a$ but does not divide $c$, and $q$ be the product of all primes that divides $mb$ but not divide $a$, if no such prime exists let corresponding p or q be $1$. Then $(p,q)=1$. If $(ap^2+bpq+cq^2,mb) \ne 1$, there exists prime $d \mid m$ and $d \mid ap^2+bpq+cq^2$. 
\begin{enumerate}[label=(\arabic*)]
	\item If $d \nmid a$, then $d \mid q$. Then $d \mid p$ since $d \mid ap^2+bpq+cq^2$ and $d \nmid a$, which implies $d \mid (p,q)$ and contradicts$(p,q)=1$.
	\item If $d \mid a,d \mid c$, we have $d \nmid p,d \nmid q$. Then $d \mid b$ since $d \mid ap^2+bpq+cq^2$, which contradicts $(a,b,c)$ is a primitive form.
	\item If $d \mid a, d\nmid c$, we have $d \mid p, d \nmid q$, which contradict $d \mid ap^2+bpq+cq^2$.
\end{enumerate}
From the three cases above we find that $(ap^2+bpq+cq^2,mb)=1$, the claim holds.

There exist $q,r \in R$ such that $ps+qr=1$ since $(p,q)=1$. If $apr+bqr+cqs \in R_b$, let $\gamma = \begin{pmatrix}p & q \\r & s\end{pmatrix}$ and \[(A,b,C)=\gamma(a,b,c)=(ap^2+bpq+cq^2,b,ar^2+brs+cs^2),\]then $(A,b,C)$ satisfies the condition. If $apr+bqr+cqs \notin R_b$, from lemma $\ref{1.9}$ there exists $v \in R $ such that $v(ap^2+bpq+cq^2) \notin R_b$. Let $\gamma= \begin{pmatrix}p & q \\r+vp & s+vq\end{pmatrix}$. Then \[\det(\gamma)=p(s+vq)+q(r+vp)=ps+qr=1\] and \[ap(r+vp)+bq(r+vp)+cq(s+vq)=apr+bqr+cqs+v(ap^2+bpq+cq^2) \in R_b.\] Then $(A,b,C)=\gamma (a,b,c)$ satisfies the condition.
\end{proof}

\begin{lemma}\label{1.11}
	Suppose $(A_1,b,C_1)$ and $(A_2,b,C_2)$ are primitive quadratic forms. Then there exists primitive forms $(a_1,b,c_1) \approx (A_1,b,C_1)$ and $(a_2,b,c_2) \approx (A_2,b,C_2)$ such that \[(a_1,b)=1,(a_2,ba_1)=1.\]
\end{lemma}
\begin{proof}
	By lemma $\ref{1.10}$ there exists primitive form $(a_1,b,c_1)\approx (A_1,b,C_1)$ such that $(a_1,b)=1$ and primitive form $(a_2,b,c_2) \approx (A_2,b,C_2)$ such that $(a_2,a_1b)=1$.
\end{proof}

\begin{lemma}\label{1.12}
	Suppose $(a_1,b,c_1)$ and $(a_2,b,c_2)$ are primitive forms and \[(a_1,b)=1,(a_2,ba_1)=1,\; a_1c_1 \equiv a_2c_2 \equiv \Delta \mod{\varphi(R)},\; a_1c_1+a_2c_2=k^2+bk,k \in R_b.\] Then there exists $e_1,e_2 \in R$, such that \[e_1a_1+e_2a_2=1,\; e_1a_1k \in R_b,\] and primitive form$(a_1a_2,b,e_1e_2k^2+c_1e_2+c_2e_1)$ is the direct composition of $(a_1,b,c_1)$ and $(a_2,b,c_2)$. Moreover, \[a_1a_2(e_1e_2k^2+c_1e_2+c_2e_1)=a_1c_1+\varphi(a_1e_1k).\]
\end{lemma}
\begin{proof}
	There exists $e_1',e_2' \in R$ such that $a_1e_1'+a_2e_2'=1$ since $(a_2,ba_1)=1$. If $a_1e_1'k \in R_b$, let $e_1=e_1',e_2=e_2'$. If $a_1e_1'k \notin R_b$, consider \[a_1(e_1'+va_2)+a_2(e_2'+va_1)=a_1e_1'+a_2e_2'=1.\]
	\begin{enumerate}[label=(\arabic*)]
		\item If $b \mid k$ and $a_2(0) \ne 0$, let $v=\frac{e_1'(0)}{a_2(0)}$, then $T \mid e_1'+va_2$, and $Tb \mid (e_1'+va_2)k$, so\\ $(e_1'+va_2)k \in R_b$. Let $e_1=e_1'+va_2,e_2=e_2'+va_1$.
		\item If $b \mid k$ and $a_2(0) =0 $, let $G$ be the subgroup of $(\mathbb{F}_{2^n},+)$ in lemma $\ref{1.1}$, then \[a_1(0)e_1'(0)=1-a_2(0)e_2'(0)=1,\; k=qb,\; q \in R, q(0) \in G.\] Then $a_1e_1'k=(a_1e_1'q)b$, $(a_1e_1'q)(0)=q(0) \in G$, so $a_1e_1'k \in R_b$, which contradicts $a_1e_1'k \notin R_b$. 
		\item  If $b\nmid k$, then $b \nmid ka_1a_2$ since $(a_1,b)=(a_2,b)=1$. From lemma $\ref{1.9}$, there exists $v \in R$ such that $va_1a_2k \notin R_b$, then \[a_1(e_1'+va_2)k=a_1e_1'k + va_1a_2k \in R_b.\] Let $e_1=e_1'+va_2,e_2=e_2'+va_1$. 
	\end{enumerate}Then we find $e_1,e_2 \in R$ such that $e_1a_1+e_2a_2=1$, $e_1a_1k \in R_b$.
	
	Let $\mathcal{A}$ be the cube
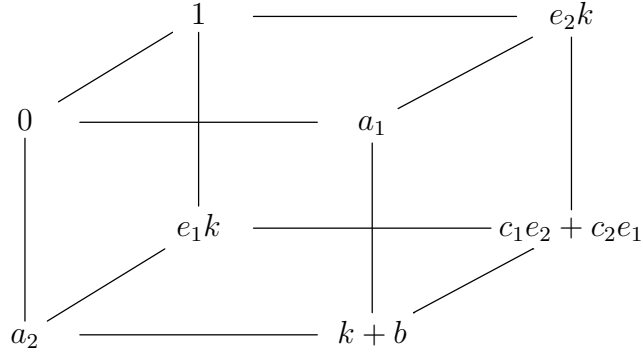
\[\xymatrix@W=3em{&
		1\ar@{-}[rr]\ar@{-}[ld]\ar@{-}[dd]& &
		e_2k\ar@{-}[ld]\ar@{-}[dd]\\
		0\ar@{-}[rr]\ar@{-}[dd]& &
		a_1\ar@{-}[dd]&\\&
		e_1k\ar@{-}[rr]\ar@{-}[ld]& &
		c_1e_2+c_2e_1\ar@{-}[ld]\\
		a_2\ar@{-}[rr]& &
		k+b&}\]
	We have\[(k+b)+(e_2k)a_2+a_1(e_1k)=k+b+k(e_1a_1+e_2a_2)=k+b+k=b,\]
	\begin{equation}\notag
		\begin{aligned}
			\begin{vmatrix}
				e_1k & c_1e_2+c_2e_1 \\
				a_2 & k+b
			\end{vmatrix}&=e_1k(k+b)+a_2(c_1e_2+c_2e_1)=e_1(a_1c_1+a_2c_2)+a_2(c_1e_2+c_2e_1)\\&=c_1(e_1a_1+e_2a_2)=c_1,
		\end{aligned}
	\end{equation}
	\begin{equation}\notag
		\begin{aligned}
			\begin{vmatrix}
				e_2k & a_1 \\
				c_1e_2+c_2e_1 & k+b
			\end{vmatrix}&=e_2k(k+b)+a_1(c_1e_2+c_2e_1)=e_2(a_1c_1+a_2c_2)+a_1(c_1e_2+c_2e_1)\\&=c_2(e_1a_1+e_2a_2)=c_2.
		\end{aligned}
	\end{equation}
	Then \begin{equation}\notag
		f_1^{\mathcal{A}}=(\begin{vmatrix}
			1 & e_2k \\
			0 & a_1
		\end{vmatrix},b,\begin{vmatrix}
		e_1k & c_1e_2+c_2e_1 \\
		a_2 & k+b
		\end{vmatrix})=(a_1,b,c_1),
	\end{equation}
	\begin{equation}\notag
		f_2^{\mathcal{A}}=(\begin{vmatrix}
			1 & 0 \\
			e_1k & a_2
		\end{vmatrix},b,\begin{vmatrix}
			e_2k & a_1 \\
			c_1e_2+c_2e_1 & k+b
		\end{vmatrix})=(a_2,b,c_2),
	\end{equation}
	\begin{equation}\notag
		f_3^{\mathcal{A}}=(\begin{vmatrix}
			0 & a_1 \\
			a_2 & k+b
		\end{vmatrix},b,\begin{vmatrix}
			1 & e_2k \\
			e_1k & c_1e_2+c_2e_1
		\end{vmatrix})=(a_1a_2,b,e_1e_1k^2+c_1e_2+c_2e_1).
	\end{equation}
	Also,\[a_1(e_1k)+(e_2k)a_2=k \in R_b, a_1(e_1k) \in R_b.\]By lemma $\ref{1.6}$ and definition $\ref{1.8}$, $(a_1a_2,b,e_1e_2k^2+c_1e_2+c_2e_1)$ is the direct composition of $(a_1,b,c_1)$ and $(a_2,b,c_2)$. We have that $(a_1a_2,b,e_1e_2k^2+c_1e_2+c_2e_1)$ is a primitive form since $(a_1,b)=(a_2,b)=1$.
	Moreover,
		\begin{align}\notag
			a_1a_2(e_1e_2k^2+c_1e_2+c_2e_1)&=a_1a_2e_1e_2k^2+a_1a_2c_1e_2+a_1a_2c_2e_1\\
			\notag&=a_1e_1(1+a_1e_1)k^2+a_1c_1(1+a_1e_1)+a_1e_1(a_1c_1+k^2+bk)\\
			\notag&=(a_1e_1k)^2+a_1c_1+a_1e_1bk\\
			\notag&=a_1c_1+\varphi(a_1e_1k).\qedhere
		\end{align}
	
\end{proof}

The next two lemma will be used to prove that direct composition can give a well-defined operation over $C(b,\Delta)$. Lemma $\ref{1.13}$ is given by Gauss to solve the similar question of quadratic forms over $\mathbf{Z}$( see lemma 8.2.3 of Fang $^{\cite[p.178]{Fang}}$), and lemma $\ref{1.14}$ aims to treat the difference between direct composition over $\mathbf{Z}$ and $\mathbb{F}_{2^n}[T]$.

\begin{lemma}\label{1.13}
	Suppose $m \ge 2$, and \[A=\begin{pmatrix}
		p_1 & p_2 & \cdots & p_m \\
		q_1 & q_1 & \cdots & q_m
	\end{pmatrix},A'=\begin{pmatrix}
	p'_1 & p'_2 & \cdots & p'_m \\
	q'_1 & q'_1 & \cdots & q'_m
	\end{pmatrix}\] satisfy\begin{enumerate}
	\item $A,A'$ have same second-order minors,
	\item All second-order minors of $A$ are coprime.
	\end{enumerate}Then there exists $\gamma \in Sl_2(R)$, such that $\gamma A=A'$.
\end{lemma}

\begin{lemma}\label{1.14}
	Suppose $\gamma=\begin{pmatrix}
		x&y\\z&w
	\end{pmatrix}\in Sl_2(R)$.
	Let \[\mathcal{A}= \begin{gathered}\xymatrix{&
		p_1\ar@{-}[rr]\ar@{-}[ld]\ar@{-}[dd]& &
		q_1\ar@{-}[ld]\ar@{-}[dd]\\
		p_2\ar@{-}[rr]\ar@{-}[dd]& &
		q_2\ar@{-}[dd]&\\&
		r_1\ar@{-}[rr]\ar@{-}[ld]& &
		s_1\ar@{-}[ld]\\
		r_2\ar@{-}[rr]& &
		s_2&}\end{gathered},f_i^{\mathcal{A}}=(a_i,b,c_i),i=1,2,3,\]
		\[^i\gamma\mathcal{A}=\begin{gathered}\xymatrix{&
		p_{1i}\ar@{-}[rr]\ar@{-}[ld]\ar@{-}[dd]& &
		q_{1i}\ar@{-}[ld]\ar@{-}[dd]\\
		p_{2i}\ar@{-}[rr]\ar@{-}[dd]& &
		q_{2i}\ar@{-}[dd]&\\&
		r_{1i}\ar@{-}[rr]\ar@{-}[ld]& &
		s_{1i}\ar@{-}[ld]\\
		r_{2i}\ar@{-}[rr]& &
		s_{2i}&}\end{gathered},i=1,2,3.\]
	Then for $i=1,2,3$, \[q_1r_2+q_2r_1+q_{1i}r_{2i}+q_{2i}r_{1i} \in R_b,p_2s_1+q_2r_1+p_{2i}s_{1i}+q_{2i}r_{1i} \in R_b\] if and only if \[a_ixz+byz+c_izw \in R_b.\] In particular, if $q_1r_2+q_2r_1,p_2s_1+q_2r_1 \in R_b$ and $\gamma_1 \in S_{(a_1,b,c_1)},\gamma_2 \in S_{(a_2,b,c_2)}$, then\\ $f_3^{^2\gamma_2^1\gamma_1\mathcal{A}}=(a_3,b,c_3)$ is a direct composition of $f_1^{^2\gamma_2^1\gamma_1\mathcal{A}}=\gamma_1(a_1,b,c_1)$ and\\ $f_2^{^2\gamma_2^1\gamma_1\mathcal{A}}=\gamma_2(a_2,b,c_2)$.
\end{lemma}
\begin{proof}
	For $i=2$, the cube $^2\gamma \mathcal{A}$ is 
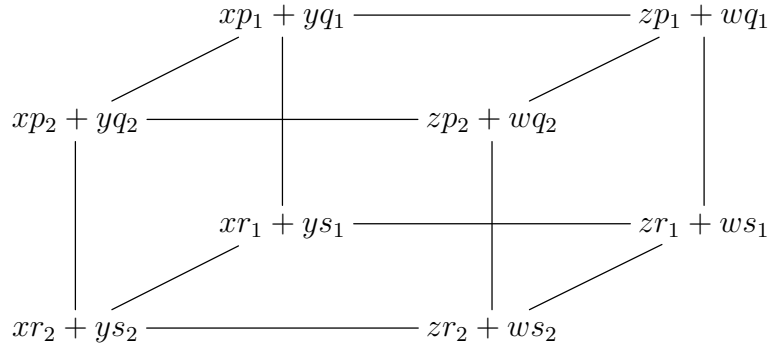
\[\xymatrix{&
		xp_1+yq_1\ar@{-}[rr]\ar@{-}[ld]\ar@{-}[dd]& &
		zp_1+wq_1\ar@{-}[ld]\ar@{-}[dd]\\
		xp_2+yq_2\ar@{-}[rr]\ar@{-}[dd]& &
		zp_2+wq_2\ar@{-}[dd]&\\&
		xr_1+ys_1\ar@{-}[rr]\ar@{-}[ld]& &
		zr_1+ws_1\ar@{-}[ld]\\
		xr_2+ys_2\ar@{-}[rr]& &
		zr_2+ws_2&}\]
	Then\begin{equation}\notag\begin{aligned}
		q_{12}r_{22}+q_{22}r_{12}&=(zp_1+wq_1)(xr_2+ys_2)+(zp_2+wq_2)(xr_1+ys_1)\\
		&=xz(p_1r_2+p_2r_1)+yz(p_1s_2+p_2s_1)+yw(q_1s_2+q_2s_1)+(1+yz)(q_1r_2+q_2r_1)\\
		&=a_2xz+byz+c_2yw+q_1r_2+q_2r_1,\\
		p_{22}s_{12}+q_{22}r_{12}&=(xp_2+yq_2)(zr_1+ws_1)+(zp_2+wq_2)(xr_1+ys_1)\\
		&=(yz+xw)(p_2s_1+q_2r_1)\\
		&=p_2s_1+q_2r_1.
	\end{aligned}\end{equation}
    So \[q_1r_2+q_2r_1+q_{12}r_{22}+q_{22}r_{12}=a_2xz+byz+c_2yw,\]\[p_2s_1+q_2r_1+p_{22}s_{12}+q_{22}r_{12}=0.\]
    Symmetrically, for $i=1$ we have\[q_1r_2+q_2r_1+q_{11}r_{21}+q_{21}r_{11}=a_1xz+byz+c_1yw,\]\[p_2s_1+q_2r_1+p_{21}s_{11}+q_{21}r_{11}=a_1xz+byz+c_1yw,\]
    and for $i=3$ we have \[q_1r_2+q_2r_1+q_{13}r_{23}+q_{23}r_{13}=0,\]\[p_2s_1+q_2r_1+p_{23}s_{13}+q_{23}r_{13}=a_3xz+byz+c_3zw.\]
    So for $i=1,2,3$, $q_1r_2+q_2r_1+q_{1i}r_{2i}+q_{2i}r_{1i} \in R_b,p_2s_1+q_2r_1+p_{2i}s_{1i}+q_{2i}r_{1i} \in R_b$ if and only if $a_ixz+byz+c_izw \in R_b$. In particular, if $q_1r_2+q_2r_1,p_2s_1+q_2r_1 \in R_b$ and $\gamma_1 \in S_{(a_1,b,c_1)},\gamma_2 \in S_{(a_2,b,c_2)}$, then $^2\gamma_2^1\gamma_1\mathcal{A}$ is direct.
\end{proof}

\begin{lemma}\label{1.15}
	Suppose primitive form $(a_1,b,c_1) \approx (A_1,b,C_1)$, $(a_2,b,c_2)\approx (A_2,b,C_2)$, $(a_3,b,c_3)$ is a direct composition of $(a_1,b,c_1)$ and $(a_2,b,c_2)$, $(A_3,b,C_3)$ is a direct composition of $(A_1,b,C_1)$ and $(A_2,b,C_2)$. Then $(a_3,b,c_3) \approx (A_3,b,C_3)$.
\end{lemma}
\begin{proof}
	Since $(a_1,b,c_1) \approx (A_1,b,C_1)$,$(a_2,b,c_2)\approx (A_2,b,C_2)$, there exists $\gamma_1 \in S_{(A_1,b,C_1)}$, $\gamma_2 \in S_{(A_2,b,C_2)}$ such that \[(a_1,b,c_1) = \gamma_1(A_1,b,C_1),\; (a_2,b,c_2)= \gamma_2(A_2,b,C_2).\]
	By lemma $\ref{1.6}$, there exists cubes \[\mathcal{A}=\begin{gathered}
		\xymatrix@R=0.8cm@C=0.8cm@H=0.4cm{&
		p_1\ar@{-}[rr]\ar@{-}[ld]\ar@{-}[dd]& &
		q_1\ar@{-}[ld]\ar@{-}[dd]\\
		p_2\ar@{-}[rr]\ar@{-}[dd]& &
		q_2\ar@{-}[dd]&\\&
		r_1\ar@{-}[rr]\ar@{-}[ld]& &
		s_1\ar@{-}[ld]\\
		r_2\ar@{-}[rr]& &
		s_2&}	\end{gathered},\quad \mathcal{A}'=\begin{gathered}\xymatrix@R=0.8cm@C=0.8cm@H=0.4cm{&
		p'_1\ar@{-}[rr]\ar@{-}[ld]\ar@{-}[dd]& &
		q'_1\ar@{-}[ld]\ar@{-}[dd]\\
		p'_2\ar@{-}[rr]\ar@{-}[dd]& &
		q'_2\ar@{-}[dd]&\\&
		r'_1\ar@{-}[rr]\ar@{-}[ld]& &
		s'_1\ar@{-}[ld]\\
		r'_2\ar@{-}[rr]& &
		s'_2&}\end{gathered}\]such that \[f_i^{\mathcal{A}}=(A_i,b,C_i),f_i^{\mathcal{A}'}=(a_i,b,c_i),i=1,2,3.\]and\[q_1r_2+q_2r_1,\; p_2s_1+q_2r_1,\; q'_1r'_2+q'_2r'_1,\; p'_2s'_1+q'_2r'_1 \in R_b.\]
	Let \[^1\gamma_1^2\gamma_2\mathcal{A}=\begin{gathered}\xymatrix@R=0.8cm@C=0.8cm@H=0.4cm{&
		p''_1\ar@{-}[rr]\ar@{-}[ld]\ar@{-}[dd]& &
		q''_1\ar@{-}[ld]\ar@{-}[dd]\\
		p''_2\ar@{-}[rr]\ar@{-}[dd]& &
		q''_2\ar@{-}[dd]&\\&
		r''_1\ar@{-}[rr]\ar@{-}[ld]& &
		s''_1\ar@{-}[ld]\\
		r''_2\ar@{-}[rr]& &
		s''_2&}\end{gathered}.\]
	Then from lemma $\ref{1.14}$ and $\gamma_1 \in S_{(A_1,b,C_1)},\gamma_2 \in S_{(A_2,b,C_2)}$ we have\[q''_1r''_2+q''_2r''_1,\; p''_2s''_1+q''_2r''_1 \in R_b.\] From equation $(\ref{5})$ we have \[f_1^{^1\gamma_1^2\gamma_2\mathcal{A}}=(a_1,b,c_1),\; f_2^{^1\gamma_1^2\gamma_2\mathcal{A}}=(a_2,b,c_2),\; f_3^{^1\gamma_1^2\gamma_2\mathcal{A}}=(A_3,b,C_3).\] Suppose $a_1c_1+a_2c_2=\varphi(k),k \in R_b$. Then by lemma $\ref{1.7}$ and $q'_1r'_2+q'_2r'_1,q''_1r''_2+q''_2r''_1 \in R_b$, \[q'_1r'_2+q'_2r'_1=q''_1r''_2+q''_2r''_1=k,\; p'_1s'_2+p'_2s'_1=p''_1s''_2+p''_2s''_1=k+b.\]Then\[M(\mathcal{A}'):=\begin{pmatrix}
	p'_2 & q'_2 & s'_2 & r'_2 \\
	p'_1 & q'_1 & s'_1 & r'_1
	\end{pmatrix} \quad \text{and}\quad M(^1\gamma_1^2\gamma_2\mathcal{A}):=\begin{pmatrix}
	p''_2 & q''_2 & s''_2 & r''_2 \\
	p''_1 & q''_1 & s''_1 & r''_1
	\end{pmatrix}\]satisfy the condition in lemma $\ref{1.13}$( all second-order minors of $M(\mathcal{A}')$ are coprime since form $(a_1,b,c_1)$ is primitive), so there exists $\gamma \in Sl_2(R)$ such that $\gamma M(\mathcal{A}')=M(^1\gamma_1^2\gamma_2\mathcal{A})$. That is,  \[^3\gamma\mathcal{A}'=^1\gamma_1^2\gamma_2\mathcal{A}.\]Then \[\gamma(a_3,b,c_3)=f_3^{^3\gamma\mathcal{A}'}=f_3^{^1\gamma_1^2\gamma_2\mathcal{A}}=(A_3,b,C_3).\]By $q'_1r'_2+q'_2r'_1,p'_2s'_1+q'_2r'_1,q''_1r''_2+q''_2r''_1,p''_2s''_1+q''_2r''_1 \in R_b$ and lemma $\ref{1.14}$ we have $\gamma \in S_{(a_3,b,c_3)}$. So $(a_3,b,c_3) \approx (A_3,b,C_3)$.
\end{proof}

\begin{theorem}\label{1.16}
	Direct composition induces an operation on the set of classes of quadratic forms. More precisely, for all $[(a_1,b,c_1)],[(a_2,b,c_2)]$ in $C(b,\Delta)$, there exists exactly one $[(a_3,b,c_3)]$ in $C(b,\Delta)$ such that $(a_3,b,c_3)$ is a direct composition of $(a_1,b,c_1)$ and $(a_2,b,c_2)$, and $[(a_3,b,c_3)]$ is not depend on the choice of representatives of $[(a_1,b,c_1)]$ and $[(a_2,b,c_2)]$. Call $[(a_3,b,c_3)]$ the direct composition of $[(a_1,b,c_1)]$ and $[(a_2,b,c_2)]$, denoted by $[(a_1,b,c_1)][(a_2,b,x_2)]=[(a_3,b,c_3)]$. Furthermore, for all $[(a,b,c)],[(A,b,C)] \in C(b,\Delta)$, we have \[[(a,b,c)][(A,b,C)]=[(A,b,C)][(a,b,c)],\]\[[(1,b,\Delta)][(a,b,c)]=[(a,b,c)],\]\[[(a,b,c)][(c,b,a)]=[(1,b,ac)]=[(1,b,\Delta)].\]
	Call $f_0:=(1,b,\Delta)$ the principal form of invariant $b,\Delta$, and call $C_0:=[(1,b,\Delta)]$ the principal class.
\end{theorem}
\begin{proof}
	From lemma $\ref{1.11}$ there exists primitive forms $(A_1,b,C_1) \approx (a_1,b,c_1)$ and \\ $(A_2,b,C_2) \approx (a_2,b,c_2)$ such that $(A_1,b)=1,(A_2,bA_1)=1$. From lemma $\ref{1.12}$ there exists $[(a_3,b,c_3)]\in C(b,\Delta)$ such that $(a_3,b,c_3)$ is a direct composition of $(A_1,b,C_1)$ and $(A_2,b,C_2)$, then $(a_3,b,c_3)$ is a direct composition of $(a_1,b,c_1)$ and $(a_2,b,c_2)$ by the particular case in lemma $\ref{1.14}$. On the other hand, if $(a'_3,b,c'_3)$ is another direct composition of $(a_1,b,c_1)$ and $(a_2,b,c_2)$, from lemma $\ref{1.15}$ we have $[(a_3,b,c_3)]=[(a'_3,b,c'_3)]$, and  $[(a_3,b,c_3)]$ is not depend on the choice of representatives of $[(a_1,b,c_1)]$ and $[(a_2,b,c_2)]$.
	
	Let \[\mathcal{A}=\begin{gathered}
		\xymatrix@R=0.8cm@C=0.8cm@H=0.4cm{&
			p_1\ar@{-}[rr]\ar@{-}[ld]\ar@{-}[dd]& &
			q_1\ar@{-}[ld]\ar@{-}[dd]\\
			p_2\ar@{-}[rr]\ar@{-}[dd]& &
			q_2\ar@{-}[dd]&\\&
			r_1\ar@{-}[rr]\ar@{-}[ld]& &
			s_1\ar@{-}[ld]\\
			r_2\ar@{-}[rr]& &
			s_2&}	\end{gathered},  \qquad
	\mathcal{A}'=\begin{gathered}
		\xymatrix@R=0.8cm@C=0.8cm@H=0.4cm{&
			p_1\ar@{-}[rr]\ar@{-}[ld]\ar@{-}[dd]& &
			r_1\ar@{-}[ld]\ar@{-}[dd]\\
			p_2\ar@{-}[rr]\ar@{-}[dd]& &
			r_2\ar@{-}[dd]&\\&
			q_1\ar@{-}[rr]\ar@{-}[ld]& &
			s_1\ar@{-}[ld]\\
			q_2\ar@{-}[rr]& &
			s_2&}	\end{gathered}.\]
	If $\mathcal{A}$ is direct then $\mathcal{A}'$ is direct by definition $\ref{1.8}$ and $p_2s_1+r_2q_1=(p_2s_1+q_2r_1)+(q_2r_1+q_1r_2)$, so from lemma $\ref{1.6}$ and \[f_1^{\mathcal{A}}=f_2^{\mathcal{A}'},\; f_2^{\mathcal{A}}=f_1^{\mathcal{A}'},\; f_3^{\mathcal{A}}=f_3^{\mathcal{A}'}\]we have \[[(a,b,c)][(A,b,C)]=[(A,b,C)][(a,b,c)].\]
	For all $[(a,b,c)] \in C(b,\Delta)$, let \[\mathcal{A}_1=\begin{gathered}
		\xymatrix@R=0.6cm@C=0.8cm@H=0.4cm{&
			1\ar@{-}[rr]\ar@{-}[ld]\ar@{-}[dd]& &
			0\ar@{-}[ld]\ar@{-}[dd]\\
			0\ar@{-}[rr]\ar@{-}[dd]& &
			1\ar@{-}[dd]&\\&
			0\ar@{-}[rr]\ar@{-}[ld]& &
			c\ar@{-}[ld]\\
			a\ar@{-}[rr]& &
			b&}	\end{gathered}  \qquad
	\mathcal{A}_2=\begin{gathered}
		\xymatrix@R=0.6cm@C=0.8cm@H=0.4cm{&
			b\ar@{-}[rr]\ar@{-}[ld]\ar@{-}[dd]& &
			a\ar@{-}[ld]\ar@{-}[dd]\\
			1\ar@{-}[rr]\ar@{-}[dd]& &
			0\ar@{-}[dd]&\\&
			c\ar@{-}[rr]\ar@{-}[ld]& &
			0\ar@{-}[ld]\\
			0\ar@{-}[rr]& &
			1&}	\end{gathered}.\]
	$\mathcal{A}_1,\mathcal{A}_2$ are direct. Then we have\[ [(1,b,ac)][(a,b,c)]= [f_1^{\mathcal{A}_1}][f_2^{\mathcal{A}_1}]=[f_3^{\mathcal{A}_1}]=[(a,b,c)], \] \[ [(a,b,c)][(c,b,a)]= [f_1^{\mathcal{A}_2}][f_2^{\mathcal{A}_2}]=[f_3^{\mathcal{A}_2}]=[(1,b,ac)].\]Finally, suppose $\Delta=ac +\varphi(k), k \in R_b$ and we have $[(1,b,ac)]=[(1,b,\Delta)]$ since \[\begin{pmatrix}
		1 & 0 \\
		k & 1
	\end{pmatrix}(1,b,ac)=(1,b,\Delta),k \in R_b.\qedhere\]
\end{proof}
If we can prove the direct composition on the set of classes of quadratic forms satisfies the associative law then direct composition will make $C(b,\Delta)$ an Abelian group. But it is complex to compute $[(a_1,b,c_1)
][(a_2,b,c_2)][(a_3,b,c_3)]$ directly, so the commutativity will be proved in next section by the correspondence between classes of quadratic forms and ideal class.

\subsection{Classes of Binary Quadratic Forms and Ideal Class} 
\qquad Let $K=Frac(R)$ be the field of fractions of $R$. Let $K (\eta)$ be a field extension of $K$, where the minimal polynomial of $\eta$ is $x^2+bx+\Delta$. We have $[K(\eta):K]=2$ since $ \Delta \not \equiv 0 \mod{\varphi(R)}$. Consider the fractional ideal of $R \oplus R\eta$.

\begin{lemma}\label{1.18}
	$R \oplus R \alpha$ is a fractional ideal of $R \oplus R\eta$ if and only if there exists $a,c \in R$, such that $\alpha$ is a root of $ax^2+bx+c$, and $ac \equiv \Delta \mod{\varphi(R)}$.
\end{lemma}
\begin{proof}
	If $a,c \in R$, $ac=\Delta +k^2+bk$, then the roots of $ax^2+bx+c$ are $\frac{\eta +k}{a}$ and $\frac{\eta+k+b}{a}$. Then $R\oplus R(\frac{\eta+k+b}{a})$ and $R \oplus R(\frac{\eta+k}{a})$ are fractional ideals of $R \oplus R\eta$.
	
	If $\alpha \in K(\eta),R \oplus R \alpha$ is a fractional ideal of $R \oplus R\eta$, suppose $\alpha = \frac{p\eta +q}{r}$, $p,q,r$ are coprime and $p$ is monic. Since $R \oplus R \alpha$ is a fractional ideal we have \[\eta \alpha = \frac{p(b\eta +\Delta)+q\eta}{r}=\frac{q^2+pqb+p^2\Delta}{pr}+\frac{pb+q}{p} \cdot \frac{p\eta +q}{r} \in R \oplus R\frac{p\eta +q}{r}.\]
	Then we have $p \mid q, pr \mid q^2+pqb+p^2\Delta$.
	Similarly, since $1 \in R$ we have \[ 1 \cdot \eta=\frac{q}{p}+\frac{r}{p}\frac{p\eta +q}{r} \in R \oplus R\frac{p\eta +q}{r},\]so $p \mid r,p \mid q$. Then $p=1$ since $p,q,r$ are coprime, and $r \mid q^2+qb+\Delta$. Then $\alpha = \frac{\eta+q}{r}$ is a root of  $rx^2+bx+\frac{\Delta +qb+q^2}{r}$.
\end{proof}

In particular, if $R\alpha \oplus R\beta$ is a fractional ideal of $R \oplus R\eta$, then $\frac{\alpha}{\beta}=\frac{\eta+k}{a},k,a \in R$, and $a \mid \Delta +k^2+bk$.

\begin{lemma}\label{new1.18}
	Suppose $R \oplus R\frac{\eta + k}{a}$ is a fractional ideal of $R \oplus R\eta$. If $R \oplus R\frac{\eta + k}{a}$ is invertible, then $(a,b,\frac{\Delta+k^2+bk}{a})$ is a primitive form.
\end{lemma}
\begin{proof}
	We claim that if there exists a prime $\tilde{p} \in R$ such that $\tilde{p}$ is a common divisor of $a,b,\frac{\Delta+k^2+bk}{a}$, then $R \oplus R\frac{\eta + k}{a}$ is non-invertible.
	Let\[\eta'=\eta +k, \Delta'=\Delta+k^2+bk.\]
	Suppose \[\alpha = \frac{p\eta'+q}{r} \in (R\oplus R\frac{\eta+k}{a})^{-1}=\{x \in K(\eta) | x(R\oplus R\frac{\eta+k}{a}) \subset R \oplus R\eta \},\]where $p,q,r \in R$, $p,q,r$ are coprime and $r$ is monic. 
	Then $\alpha \cdot 1 \in R\oplus R\eta = R \oplus R \eta'$ implies $r=1$, and \begin{equation}\notag
		\alpha\frac{\eta+k}{a}=\frac{p(b\eta'+\Delta')+q\eta'}{a}
		=\frac{p\Delta'}{a}+\frac{pb+q}{a}\eta'
		\in R\oplus R \eta=R \oplus R\eta'
	\end{equation}
	implies\[a\mid p\Delta', \: a \mid pb+q.\]
	Then $\tilde{p} \mid q$ since $\tilde{p} \mid a, \tilde{p} \mid b$. Also, we have $\tilde{p} \mid \frac{p\Delta'}{a}$. Then \[\alpha(R \oplus R\frac{\eta + k}{a}) \subset \tilde{p}R \oplus R\eta' \subsetneq R\oplus R\eta,\] so $R \oplus R\frac{\eta + k}{a}$ is non-invertible.
\end{proof}

\begin{definition}\label{1.17}
	Suppose $M,M'$ are fractional ideal of $R \oplus R\eta$. Call $M$ and $M'$ are equivalent, if there exists $\alpha \in K(\eta)$ such that $M=\alpha M'$. Denote the equivalence class of $M$ by $[M]$. Let \[C_2(b,\Delta):=\{[R \oplus R\frac{\eta + k}{a}]\;|\; a\mid \Delta+k^2+bk, \text{ and } a,b,\frac{\Delta+k^2+bk}{a} \text{ are coprime}\},\] and \[\text{Pic}(R \oplus R\eta):= \{[M]\;|\; M \text{is a invertible fractional ideal of }R \oplus R\eta \}.\] From lemma $\ref{1.18}$ and lemma $\ref{new1.18}$ we have $\text{Pic}(R \oplus R \eta) \subset C_2(b,\Delta)$, and we will show that $\text{Pic}(R \oplus R \eta) = C_2(b,\Delta)$.
\end{definition}

In $\mathbf{Z}$ we have the the group of classes of quadratic forms with discriminant $D$ over $\mathbf{Z}$ whose multiplication is given by Gauss's direct composition and the narrow ideal class group of \linebreak $\mathbf{Z} \oplus \mathbf{Z} \frac{D+\sqrt{D}}{2}$ are isomorphic. We have similar result in $R=\mathbb{F}_{2^n}[T]$:

\begin{theorem}\label{1.19}
	Suppose $[R \alpha \oplus R \beta] \in C_2(b,\Delta)$,  $\frac{\alpha}{\beta}=\frac{\eta+k}{a},a,k \in R, a \mid \Delta +k^2+bk$. Then $\frac{\beta}{\alpha}=\frac{a}{\eta+k}=\frac{\eta +k+b}{(\frac{\Delta+k^2+bk}{a})}, \frac{\Delta+k^2+bk}{a} \in R$, and there is exactly one of $k,k+b \in R_b$. Let \begin{equation}\notag
		\Pi([R \alpha \oplus R \beta]):=\begin{cases}
			[(a,b,\frac{\Delta+k^2+bk}{a})],&\text{if } k \in R_b ; \\
			[(\frac{\Delta+k^2+bk}{a},b,a)],&\text{if } k+b \in R_b.
		\end{cases}
	\end{equation}
	Then $\Pi$ is a well-defined bijection from $C_2(b,\Delta)$ to $C(b,\Delta)$. Moreover,  \[C_2(b,\Delta)=\text{Pic}(R \oplus R \eta),\] and \[\Pi([M_1M_2])=\Pi([M_1])\Pi([M_2]), \forall [M_1],[M_2] \in C_2(b,\Delta),\] in the right-hand side the multiplication is direct composition defined in theorem $\ref{1.16}$.
\end{theorem}
\begin{proof}
	Suppose $[R \alpha_1 \oplus R \beta_1]=[R \alpha_2 \oplus R \beta_2]$, and $\frac{\alpha_i}{\beta_i}=\frac{\eta+k_i}{a_i}$, $k_i \in R_b$, $i=1,2$. Let \linebreak $c_i:=\frac{\Delta+k_i^2+bk_i}{a_i}$, $i=1,2$. There exists $\gamma=\begin{pmatrix}
		p&q\\r&s
	\end{pmatrix}  \in Sl_2(R)$ and $ \xi \in R$ such that \linebreak $ \xi \begin{pmatrix}
		\beta_1 \\ \alpha_1 
	\end{pmatrix}=\gamma\begin{pmatrix}
		\beta_2 \\ \alpha_2
	\end{pmatrix}$. Then \begin{equation}\notag
		\frac{\alpha_1}{\beta_1}=\frac{r\beta_2+s\alpha_2}{p\beta_2+q\alpha_2}
		=\frac{s\frac{\eta +k_2}{a_2}+r}{q\frac{\eta +k_2}{a_2}+p},
	\end{equation}\begin{equation}\label{13}
		\begin{aligned}
			\frac{s\frac{\eta +k_2}{a_2}+r}{q\frac{\eta +k_2}{a_2}+p}
			&=\frac{(s\eta+sk_2+a_2r)(q\eta +qk_2+a_2p+bq)}{(q\eta +qk_2+a_2p)(q\eta +qk_2+a_2p+bq)}\\
			&=\frac{a_2(qr+ps)\eta+a_2(ps+qr)k_2+a_2^2pr+a_2bqr+qs(\Delta+k_2^2+bk_2)}{a_2^2p^2+a_2bpq+q^2(\Delta+k_2^2+bk_2)}\\
			&=\frac{\eta+k_2+a_2pr+bqr+c_2qs}{a_2p^2+bpq+c_2q^2}.
		\end{aligned}
	\end{equation}
	Notice that equation $(\ref{13})$ only need $ps+qr=1,a_2c_2=\Delta+k_2^2+bk_2$ and $\eta^2+b\eta+\Delta=0$. Then $k_1=k_2+a_2pr+bqr+c_2qs,a_1=a_2p^2+bpq+c_2q^2,c_1=\frac{\Delta+\varphi(k_1)}{a_1}=\frac{a_2c_2+\varphi(a_2pr+bqr+c_2ps)}{a_1}$. That is, \[ (a_1,b,c_1)=(a_2p^2+bpq+c_2q^2,b,\frac{a_2c_2+\varphi(a_2pr+bqr+c_2ps)}{a_1})=\gamma(a_2,b,c_2),\]and \[ a_2pr+bqr+c_2qs=k_1+k_2 \in R_b. \] So $(a_2,b,c_2) \approx (a_1,b,c_1)$, $\Pi$ is well-defined. Moreover, for all $[M] \in C_2(b,\Delta)$ there exists $k,a \in R$ such that $[M]=[R \oplus R \frac{\eta+k}{a}]$, $a \mid \Delta+k^2+bk$ and $a,b,\frac{\Delta+k^2+bk}{a}$ are coprime, that is, $(a,b,\frac{\Delta+k^2+bk}{a})$ is a primitive form. Then $\Pi([M])=\Pi([R \oplus R\frac{\eta+k}{a}]) \in C_(b,\Delta)$.
	
	Since $\Pi$ is well-defined and $[R \alpha \oplus R\beta]=[R\oplus R \frac{\alpha}{\beta}]=[R\oplus R\frac{\beta}{\alpha}]$, we only consider the ideal class of the form $[R \oplus R\frac{\eta+k}{a}],k,a \in R, k\in R_b$. If $\Pi([R\oplus R \frac{\eta+k_1}{a_1}])=\Pi([R\oplus R \frac{\eta+k_2}{a_2}])$, let \[c_i:=\frac{\Delta+k_i^2+bk_i}{a_i},i=1,2,\] then there exists $\gamma=\begin{pmatrix}
		p&q\\r&s
	\end{pmatrix}\in S_{(a_2,b,c_2)} $such that $(a_1,b,c_1)=\gamma(a_2,b,c_2)$, so \[a_1=a_2p^2+bpq+c_2q^2,\; a_1c_1+a_2c_2=\varphi(a_2pr+bqr+c_2qs).\]
	From \[a_1c_1+a_2c_2=(\Delta+\varphi(k_1)+(\Delta+\varphi(k_2))=\varphi(k_1+k_2)\] and $k_1,k_2 \in R_b, a_2pr+bqr+c_2qs \in R_b$ we have $k_1+k_2=a_2pr+bqr+c_2qs$. From equation $(\ref{13})$ we find \[\frac{s\frac{\eta +k_2}{a_2}+r}{q\frac{\eta +k_2}{a_2}+p}=\frac{\eta+k_1}{a_1},\]
	then from $\gamma \in Sl_2(R)$ we have \[ [R \oplus R \frac{\eta+k_1}{a_1}]=[R \oplus R\frac{s\frac{\eta +k_2}{a_2}+r}{q\frac{\eta +k_2}{a_2}+p}]=[R(q\frac{\eta +k_2}{a_2}+p) \oplus R(s\frac{\eta +k_2}{a_2}+r)]=[R \oplus R \frac{\eta +k_2}{a_2}].\] Thus $\Pi$ is injective.
	
	Suppose $[(a,b,c)] \in C(b,\Delta),ac=\Delta+\varphi(k),k\in R_b$. Then $\frac{\eta+k}{a}$ is a root of $ax^2+bx+c$. \linebreak From lemma $\ref{1.18}$ and $(a,b,c)$ is a primitive form we have $[R \oplus R \frac{\eta+k}{a}] \in C_2(b,\Delta)$, and \linebreak $\Pi([R\oplus R\frac{\eta+k}{a}])=([a,b,c])$. So $\Pi$ is surjective.
	
	Suppose $[(A_1,b,C_1)],[(A_2,b,C_2)] \in C(b,\Delta)$. From lemma $\ref{1.11}$ there exists \\ $(a_1,b,c_1) \approx (A_1,b,C_1), (a_2,b,c_2) \approx (A_2,b,C_2)$, such that $(a_1,b)=(a_2,ba_1)=1$. Suppose \[a_1c_1=\Delta + \varphi(k'),\; a_2c_2+a_1c_2=\varphi(k),\; k',k \in R_b.\]From lemma $\ref{1.12}$ there exists $e_1,e_2 \in R$ such that $a_1e_1+a_2e_2=1, e_1a_1k \in R_b$. Let $\eta'=\eta+k'$, then we have \[[(a_1,b,c_1)][(a_2,b,c_2)]=[(a_1a_2,b,e_1e_2k^2+c_1e_2+c_2e_1)],\]\[a_1a_2(e_1e_2k^2+c_1e_2+c_2e_1)=a_1c_1+\varphi(a_1e_1k).\] From the proof of surjectivity above we have\[\Pi^{-1}([(a_1,b,c_1)])=[R \oplus R \frac{\eta'}{a_1}],\; \Pi^{-1}([(a_2,b,c_2)])=[R \oplus R \frac{\eta'+k}{a_2}],\]\[\Pi^{-1}[(a_1a_2,b,e_1e_2k^2+c_1e_2+c_2e_1)]=[R\oplus R\frac{\eta'+a_1e_1k}{a_1a_2}] \in C_2(b,\Delta).\]
	On the other hand, we have \[e_2\frac{\eta'}{a_1}+e_1\frac{\eta'+k}{a_2}=\frac{\eta'+a_1e_1k}{a_1a_2},\]\[a_2\frac{\eta'+a_1e_1k}{a_1a_2}+e_1k=\frac{\eta'}{a_1},\; a_1\frac{\eta'+a_1e_1k}{a_1a_2}+e_2k=\frac{\eta'+k}{a_2},\]\[(k+b)\frac{\eta'+a_1e_1k}{a_1a_2}+c_1c_2+e_1c_2=\frac{\eta'}{a_1}\frac{\eta'+k}{a_2},\]so \[(R \oplus R \frac{\eta'}{a_1})(R \oplus R \frac{\eta'+k}{a_2})=R\oplus R\frac{\eta'+a_1e_1k}{a_1a_2}.\]
	From the arbitrariness of $[(A_1,b,C_1)],[(A_2,b,C_2)]$ and $\Pi$ is bijective we have $C_2(b,\Delta)$ is closed under multiplication and $\Pi([M_1M_2])=\Pi([M_1])\Pi([M_2]), \forall [M_1],[M_2] \in C_2(b,\Delta)$. Moreover, from theorem 1.16, for all $[M] \in C_2(b,\Delta)$, suppose $\Pi([M])=[(a,b,c)]$, then we have $[M]$ is invertible since\[[M][\Pi^{-1}([(c,b,a)])]=\Pi^{-1}([(a,b,c)][(c,b,a)])=\Pi^{-1}([(1,b,\Delta)])=[R \oplus R\eta].\] Then $C_2(b,\Delta) \subset \text{Pic}(R\oplus R\eta)$. From lemma $\ref{1.18}$ and lemma $\ref{new1.18}$ we have \[\text{Pic}(R \oplus R \eta) = C_2(b,\Delta).\qedhere\]
\end{proof}

\begin{theorem}\label{1.22}
	$C(b,\Delta)$ is a finite Abelian group under direct composition. 
\end{theorem}
\begin{proof}
	From theorem $\ref{1.16}$ we only need to show that direct composition satisfies the associative law and $C(b,\Delta)$ is finite. The associativity follows from theorem $\ref{1.19}$ since $C_2(b,\Delta)$ satisfies the associative law.
	To show the finiteness we claim that if $[(a,b,c)]\in C(b,\Delta)$ satisfies \[\deg(ac)> \max\{ \deg\Delta, 2\deg b\},\] then there exists $(a',b,c') \approx (a,b,c)$ such that $\deg(a'c') < \deg(ac)$.
	
	Suppose $[(a,b,c)]\in C(b,\Delta)$, $ac=\Delta + k^2 +bk$. If $\deg (ac) > \max\{ \deg\Delta, 2\deg b\}$, then $\deg(k^2+bk) > \max\{ \deg\Delta, 2\deg b\}$. Then $\deg(k) > \deg(b)$,  $\deg(k^2) > \deg(bk) > 2\deg(b)$. So \[\deg(a)+\deg(c)= \deg(ac)=\deg(\Delta + k^2 +bk)=\deg(k^2)=2\deg(k).\] Without loss of generality, suppose $\deg (c) = \deg(a) + 2m, m \in \mathbf{N}$, then $\deg(c)-m > \deg(b)$ since $\deg(a)+\deg(c)>2\deg(b)$. There exists $u \in \mathbb{F}_{2^n}$ such that $\deg(c+a(uT^m)^2) < \deg (c)$ since every element in a finite field of characteristic 2 has square root. Since \[\deg(b\cdot uT^m)=\deg(b)+m < \deg(c),\] we have
	\[\deg(c+b\cdot uT^m+a(uT^m)^2))< \max\{\deg(c+a(uT^m)^2),\deg(b\cdot uT^m)\}<\deg(c).\]
	Let $\gamma_1=\begin{pmatrix}
		1&0\\uT^m& 1
	\end{pmatrix}$,$\gamma_2=\begin{pmatrix}
	uT^m&1\\1&0
	\end{pmatrix}$. Then there is exactly one of $\gamma_1,\gamma_2 \in S_{(a,b,c)}$, then $(a,b,c+b\cdot uT^m+a(uT^m)^2) \approx (a,b,c)$ or $(c+b\cdot uT^m+a(uT^m)^2,b,a) \approx (a,b,c)$. Since 
	\begin{equation}\notag\begin{aligned}
	\deg(a(c+b\cdot uT^m+a(uT^m)^2))&=\deg(a)+\deg(c+b\cdot uT^m+a(uT^m)^2)\\
	&<\deg(a)+\deg(c)=\deg(ac),
		\end{aligned}
	\end{equation}
	the assertion holds. Thus, through a finite number of steps, for all $[(a,b,c)] \in C(b,\Delta)$ we can find $(A,b,C) \approx (a,b,c)$ such that $\deg(AC) \le \max\{ \deg\Delta, 2\deg b\}$. The polynomials in $F_{2^n}[T]$ of degree at most $\max\{ \deg\Delta, 2\deg b\}$ is finite, so $C(b,\Delta)$ is finite. 
\end{proof}

\begin{example}\label{eg1}
	Let $n=2$, that is, $R=\mathbb{F}_2[T]$. From the proof of theorem $\ref{1.22}$, to find all the elements in $C_(b,\Delta)$ we only need to find all the polynomials in $\Delta+\varphi(R)$ of degree less than or equal to $\max \{\deg(\Delta),2\deg(b)\}$.
\begin{enumerate}
	\item Let $b_1=\Delta_1=1$, then $C_(b_1,\Delta_1)=[(1,1,1)]$ is a trivial group. On the other hand, let $\eta_1$ be a root of $x^2+x+1=0$ then $K(\eta_1)=\mathbb{F}_4(T)$, and all the fractional ideal of $R \oplus R\eta_1 = \mathbb{F}_4[T]$ are equivalent to $\mathbb{F}_4[T]$, so $C_2(b_1,\Delta_1)$ is a trivial group.
	\item Let $b_2=T^2,\Delta_2=T^4$, and $R_b$ be the set of all the polynomials in $R$ without quadratic term, then \[\{\Delta' \in \Delta+\varphi(R) | \deg(\Delta') \le 4\}=
	\{T^4,T^4+T^2+1,T^4+T^3+T^2,T^4+T^3+1\}.\]Consider the factorization of the polynomials above and remove forms equivalent to the principal form $(1,T^2,T^4)$ we find
	\begin{equation}\notag\begin{aligned}
			C(b_2,\Delta_2)=\{&[(1,T^2,T^4)],[(T^2+T+1,T^2,T^2+T+1)],[(T^2,T^2,T^2+T+1)],\\&[(T^2+T+1,T^2,T^2)]\}.
		\end{aligned}\end{equation}
	We have $[(1,T^2,T^4)] \ne [(T^2+T+1,T^2,T^2+T+1)]$ since for all $p,q \in R$ we have $p^2+T^2pq+T^4q^2 \ne T^2+T+1$. From 
	\[\begin{pmatrix}
		1&1\\1&0
	\end{pmatrix}(T^2,T^2,T^2+T+1)=(T^2+T+1,T^2,T^2), T^2+T^2=0 \in R_b,\]
	\[ \begin{pmatrix}
	0&1\\1&1
	\end{pmatrix}(T^2+T+1,T^2,T^2+T+1)=(T^2+T+1,T^2,T^2), T^2+T^2+T+1=T+1 \in R_b,\]
	we have \[(T^2+T+1,T^2,T^2+T+1)\approx(T^2,T^2,T^2+T+1)\approx (T^2+T+1,T^2,T^2).\]
	Then $C(b_2,\Delta_2)=\{[(1,T^2,T^4)],[(T^2+T+1,T^2,T^2+T+1)]\} \cong \mathbf{Z}/2\mathbf{Z}$. Let $\eta_2$ be a root of $x^2+T^2x+T^4=0$, then from theorem $\ref{1.19}$ we have \[\text{Pic}(R\oplus R\eta_2)=C_2(b_2,\Delta_2)=\{ [R\oplus R \eta_2],[R\oplus R \frac{\eta_2+1}{T^2+T+1}]\} \cong \mathbf{Z}/2\mathbf{Z}.\] Notice that $K(\eta_2)=K(\frac{\eta_2}{T^2})=K(\eta_1)$ but $C_2(b_2,\Delta_2)$ and $C_2(b_1,\Delta_1)$ are not isomorphic since we consider the invertible fractional ideal of different ring.
	\item Let $b_3=1,\Delta_3=T^3$. Let $R_b$ be the set of all the polynomials in $R$ without constant term. Similar to the method in (2), we have\begin{equation}\notag
		C(b_3,\Delta_3)=\{[(1,1,T^3)],[(T,1,T^2)],[T^2,1,T],[(T,1,T^2+T+1)],[(T^2+T+1,1,T)]\}.
	\end{equation}
	From\[\begin{pmatrix}
		1&0\\1&1
	\end{pmatrix}(T,1,T^2)=(T,1,T^2+T+1),\; \begin{pmatrix}
	1&0\\1&1
	\end{pmatrix}(T^2,1,T)=(T^2+T+1,1,T)\]
	we have \[(T,1,T^2)\approx(T,1,T^2+T+1),\; (T^2,1,T)\approx(T^2+T+1,1,T).\]
	The cube $\mathcal{A}$ in example $\ref{eg0}$ is a direct cube, so
	we have \[ [(T,1,T^2)]^2=[(T^2,T,1)]=[(T,1,T^2)]^{-1}.\] So\[C(b_3,\Delta_3)=\{[(1,1,T^3)],[(T,1,T^2)],[(T^2,1,T)]\}\cong \mathbf{Z}/3\mathbf{Z}.\]
	
\end{enumerate}
\end{example}

\begin{myremark}
	From the method in lemma $\ref{1.15}$, we can show that if $[(a_1,b,c_1)],[(a_2,b,c_2)] \in C(b,\Delta)$ and $(a_3,b,c_3)$ is a composition of $(a_1,b,c_1)$ and $(a_2,b,c_2)$, then $(a_3,b,c_3)$ belongs to one of $[(a_1,b,c_1)][(a_2,b,c_2)]$, $[(c_1,b,a_1)][(a_2,b,c_2)]$, $[(a_1,b,c_1)][(c_2,b,a_2)]$, $[(c_1,b,a_1)][(c_2,b,a_2)]$. There is a same result in $\mathbf{Z}$. Gauss proved that if $f_1,f_2$ are forms in $\mathbf{Z}$ of discriminant $D$, then the composition of $f_1$ and $f_2$ defined by $\ref{0.1}$ belongs to four proper equivalent class: $[f_1][f_2]$, $[f_1]^{-1}[f_2]$, $[f_1][f_2]^{-1}$, $[f_1]^{-1}[f_2]^{-1}$, where $[f]$ denote the proper equivalent class of $f$ and $[f_1][f_2]$ is the Gauss composition. So if all the elements in $C(b,\Delta)$ has order 1 or 2, direct composition is exactly the same as composition. But as shown in example $\ref{eg0}$ and $\ref{eg1}$(3), we have to consider direct composition. 
	
\end{myremark}

\section{Gauss Genus Theory in R}
\qquad Fix $b \in R$, $b \ne 0$, $\Delta \in R$, $\Delta \not \equiv 0 \mod{\varphi(R)}$.
 
\subsection{Generalization of Gauss's Map in $R$}
\qquad Call $m \in R$ represented by $(a,b,c)$ if there exists $x,y \in R$ such that $ax^2+bxy+cy^2=m$. \linebreak Call $[m] \in (R/(b^2))^*$ represented by form $(a,b,c)$ if there exists an element in $[m]$ represented by $(a,b,c)$. Next lemma is similar to the result over $\mathbf{Z}$( see lemma 8.2.5 and lemma 8.3.1 of Fang$^{\cite[p.182,194]{Fang}}$).

\begin{lemma}\label{2.1}
	\begin{enumerate}[label=(\arabic*)]
		\item If $(a_1,b,c_1) \approx (a_2,b,c_2)$, then $m \in R$ or $[m] \in (R/(b^2))^*$ is represented by $(a_1,b,c_1)$ if and only if $m$ or $[m]$ is represented by $(a_2,b,c_2)$. Call $m$ or $[m]$ represented by class $[(a,b,c)]$ if $m$ or $[m]$ is represented by $(a,b,c)$. 
		\item If $m_1$ and $m_2$ are represented by $[(a_1,b,c_1)]$ and $[(a_2,b,c_2)]$ respectively, then $m_1m_2$ is represented by $[(a_1,b,c_1)][(a_2,b,c_2)]$.
		\item The values in $(R / (b^2))^*$ represented by the principal class $C_0=[(1,b,\Delta)]$ form a subgroup $H$ of $(R / (b^2))^*$. The values in $(R / (b^2))^*$ represented by $C\in C(b,\Delta)] $ form a coset of $H$, denoted by $H_{C}$.
	\end{enumerate}
\end{lemma}
\begin{proof}
	$(1)$ follows from the definition of proper equivalence. $(2)$ follows from equation $\ref{3}$. Then, since $C_0^2=C_0$, we have $H$ is closed under multiplication. $(R/(b^2))^*$ is a finite group, so $H$ is a finite subgroup of $(R/(b^2))^*$. $\forall C \in C(b,\Delta)$, from the proof of lemma $\ref{1.10}$ we have $ H_C \ne \emptyset$. Suppose $[m]\in H_C$ and $[m'] \in H_{C^{-1}}$. Then by $CC_0=C$ and $C^{-1}C=C_0$ we have $[m]H \subset H_C$ and $[m']H_C \subset H$, so $\#H_C = \#H$ and $H_C = [m]H$. 
\end{proof}

\begin{definition}\label{2.2}
	By lemma $\ref{2.1}$, we have group homomorphism
	\begin{equation}\notag
		\begin{aligned}
			\omega: C(b,\Delta) &\to (R/(b^2))^*/H\\
			C&\mapsto H_C
		\end{aligned}
	\end{equation}
\end{definition}

\begin{myremark}
	In fact $b^2$ is exactly $b^2-4ac$ in characteristic 2. That is, although we have two invariant $b$ and $ac \mod{\varphi(R)}$, which are differ from discriminant over $\mathbf{Z}$, but the generalized Gauss's map is the same as original definition.  
\end{myremark}

\subsection{The Kernel of $\omega$}
\begin{lemma}\label{2.3}
	The group homomorphism \[\omega: C(b,\Delta) \to (R/(b^2))^*/H\]induces a group homomorphism\[\bar{\omega}: C(b,\Delta)/C(b,\Delta)^2 \to (R/(b^2))^*/H.\]
\end{lemma}
\begin{proof}
	For all $C \in C(b,\Delta)$, choose $[m] \in H_C$. Then $[m^2]\in H_{C^2}$. On the other hand, $[m^2] \in H$ since $m^2=f_0(m,0)$. So $H_{C^2}=H$, that is, $C^2 \in \ker(\omega)$. Then $C(b,\Delta)^2 \subset \ker(\omega)$, and $\omega$ induces a group homomorphism \[\bar{\omega}: C(b,\Delta)/C(b,\Delta)^2 \to (R/(b^2))^*/H.\qedhere\]
\end{proof}

We have proved that $C(b,\Delta)^2 \subset \ker(\omega)$, and in this subsection we will show that \linebreak $C(b,\Delta)^2 = \ker(\omega)$, that is, $\bar{\omega}$ is injective. To prove this we need ternary quadratic form. A ternary quadratic form over $R$ is a fuction of the form\[f(x,y,z)=ax^2+by^2+cz^2+uyz+vxz+wxy,\quad a,b,c,u,v,w \in R.\]
Call two ternary quadratic form $f$ and $g$ are equivalent, if there exists \[\gamma=\begin{pmatrix}
	p_1 & p_2 & p_3\\q_1 &q_2 &q_3 \\r_1 &r_2 &r_3
\end{pmatrix} \in Sl_3(R),\]such that \begin{equation}\label{14}\begin{aligned}
	g(x,y,z)=&f((x,y,z)\gamma)\\
	=&f(p_1x+q_1y+r_1z,p_2x+q_2y+r_2z,p_3x+q_3y+r_3z)\\
	=&f(p_1,p_2,p_3)x^2+f(q_1,q_2,q_3)y^2+f(r_1,r_2,r_3)z^2\\
	&+\begin{vmatrix}
		u & v & w\\q_1 &q_2 &q_3 \\r_1 &r_2 &r_3
	\end{vmatrix}yz+\begin{vmatrix}
	p_1 & p_2 & p_3\\u &v &w \\r_1 &r_2 &r_3
	\end{vmatrix}xz+\begin{vmatrix}
	p_1 & p_2 & p_3\\q_1 &q_2 &q_3 \\u &v &w
	\end{vmatrix}xy.
\end{aligned}
\end{equation}
Denoted by $g=\gamma f$ and $g \approx f$, by definition we have $(\gamma_1\gamma_2)f=\gamma_1(\gamma_2f)$ and $g=\gamma f$ implies $\gamma^{-1}g=f$.

The proof of next result is adapted from theorem 8.3.5 of Fang$^{\cite[pp.209-211]{Fang}}$.
\begin{theorem}
	The kernel of $\omega:C(b,\Delta) \to (R/(b^2))^*/H$ is $C(b,\Delta)^2$.
\end{theorem}
\begin{proof}
	Suppose $\omega([(a,b,c)])=H$, then there exists $p,q,w \in R$ such that \[ap^2+bpq+cq^2=wb^2+1.\]Consider the ternary form\[f(x,y,z)=ax^2+cy^2+wz^2+pyz+qxz+bxy.\]
	Let $e_1=\frac{ap}{(ap,cq)}$,$e_2=\frac{cq}{(ap,cq)}$, then there exists $h_1,h_2 \in R$ such that $e_1h_1+e_2h_2=pq+wb$. Let\[\gamma=\begin{pmatrix}
	p&q&b\\h_1&h_2&(ap,cq)\\e_2&e_1& 0
	\end{pmatrix}.\]Then\[\det(\gamma)=pe_1(ap,cq)+qe_2(ap,cq)+b(e_1h_1+e_2h_2)=ap^2+cq^2+bpq+wb^2=1,\]so $\gamma \in Sl_3(R)$. From equation $\ref{14}$ we have\begin{equation}\notag\begin{aligned}
	\gamma f=&f(p,q,b)x^2+f(h_1,h_2,(ap,cq))y^2+f(e_2,e_1,0)z^2\\
	&+\begin{vmatrix}
		p&q&b\\h_1&h_2&(ap,cq)\\e_2&e_1& 0
	\end{vmatrix}yz+\begin{vmatrix}
	p&q&b\\p&q&b\\e_2&e_1& 0
	\end{vmatrix}xz+\begin{vmatrix}
	p&q&b\\h_1&h_2&(ap,cq)\\p&q&b
	\end{vmatrix}xy\\
	=&(ap^2+bq^2+wb^2+pqb+qpb+bpq)x^2+f(h_1,h_2,(ap,cq))y^2+f(e_2,e_1,0)z^2+yz\\
	=&x^2+b_1y^2+c_1z^2+yz,\text{ where }b_1=f(h_1,h_2,(ap,cq)),c_1=f(e_2,e_1,0).
	\end{aligned}
	\end{equation}
	Let \[f_1(x,y,z)=x^2+b_1y^2+c_1z^2+yz.\] We claim that if $\deg(b_k)+ \deg(c_k) \neq 0$, we can find \[f_{k+1}(x,y,z)=x^2+b_{k+1}y^2+c_{k+1}z^2+yz\] such that $f_{k+1} \approx f_k$, and $\deg(b_{k+1})+\deg(c_{k+1}) < \deg(b_k)+ \deg(c_k)$. 
	\begin{enumerate}[label=(\arabic*)]
		\item If $2 \mid \deg(b_k)$，then since every element in $\mathbb{F}_{2^n}$ has a square root we can find $u \in R$ such that $\deg(u^2+b_k)<\deg(b_k)$. Let \[f_{k+1}(x,y,z)=f_k(x+uy,y,z)=x^2+(u^2+b_k)y^2+c_kz^2+yz.\]
		\item Similarly, if $2 \mid \deg(c_k)$, choose $u \in R$ such that $\deg(u^2+c_k) < \deg(c_k)$ and let \[f_{k+1}(x,y,z)=f_k(x+uz,y,z)=x^2+b_ky^2+(u^2+c_k)z^2+yz.\]
		\item If $2\nmid \deg(b_k)$ , $ 2 \nmid \deg(c_k)$, without loss of generality, suppose $\deg(b_k) \le \deg(c_k)$. There exists $u \in R,\deg(u)=\frac{1}{2}(\deg(c_k)-\deg(b_k))$ such that $\deg(b_ku^2+c_k) < \deg(c_k)$ since $2 \mid (\deg(c_k)-\deg(b_k))$. Let\[f_{k+1}(z,y,z)=f_k(x,y+uz,z)=x^2+b_ky^2+(b_ku^2+c_k+u)z^2+yz.\]Then \[\deg(b_ku^2+c_k+u) \le \max\{\deg(b_ku^2+c_k),\deg(u)\} < \deg(c_k).\]
	\end{enumerate}
	Then the claim follows. Thus, by finite steps, we have \[f_1 \approx f_2 \approx \cdots \approx f_m,\: f_m(x,y,z)=x^2+b_my^2+c_mz^2+yz,\: b_m,c_m \in \mathbb{F}_{2^n}.\]
	Let\[g(x,y,z)=f_m(x+\sqrt{b_m}y+\sqrt{b_m}c_mz,\: y+c_mz, \: z)=x^2+yz.\]Then we have \[f\approx f_1 \approx \cdots \approx f_m \approx g.\] 
	Thus, there exists $\gamma = \begin{pmatrix}
		p_1 & p_2 & p_3\\q_1 &q_2 &q_3 \\r_1 &r_2 &r_3
	\end{pmatrix} \in Sl_3(R),$ such that \[f(x,y,z)=g((x,y,z)\gamma).\]
	Let $z=0$, then \[f(x,y,0)=ax^2+bxy+cy^2=(p_1x+q_1y)^2+(p_2x+q_2y)(p_3x+q_3y).\]
	Compare the coefficient of $xy$ we find $p_2q_3+q_2p_3=b$.
	Let \[L_1=p_1x+q_1y,\; L_2=p_2x+q_2y,\; L_3=p_3x+q_3y.\]
	For all $t \in R$, we have\begin{equation}\notag
	f(x,y,0)=L_1^2+L_2L_3=[L_1+(t^2+t)L_2+L_3]^2+[(t+1)^2L_2+L_3](t^2L_2+L_3)
	\end{equation}
	Let $x=t^2q_2+q_3$, $y=t^2p_2+p_3$, then\[L_1=t^2(p_1q_2+p_2q_1)+(p_1q_3+p_3q_1),\: L_2=p_2q_3+q_2p_3,\: L_3=t^2(p_2q_3+q_2p_3),\]
	\begin{equation}\notag\begin{aligned}
			f(t^2q_2+q_3,t^2p_2+p_3,0)&=[(p_1q_2+p_2q_1)t^2+(p_2q_3+p_3q_2)t+(p_1q_3+p_3q_1)]^2\\
			&=(At^2+bt+C)^2, \text{ where }A=p_1q_2+p_2q_1, C=p_1q_3+p_3q_1
		\end{aligned}\end{equation}
	We claim that there exists $t \in R$ such that $(At^2+bt+C,b)=1$. From Chinese remainder theorem we only need to show for all prime $p \mid b$ there exists $t \in R$ such that $(At^2+bt+C,b)=1$. Since $\gamma \in Sl_3(R)$, we have\[\det(\gamma)=r_1(p_2q_3+q_2p_3)+r_2(p_1q_3+p_3q_1)+r_3(p_1q_2+p_2q_1)=r_1b+r_2C+r_3A=1.\]
	Then $A,b,C$ are coprime. If $p \mid C$, then $p \nmid A$, let $t=1$. If $p \nmid C$, let $t=0$. The claim follows. Then there exists $x=t^2q_2+q_3$, $y=t^2p_2+p_3$ such that \[f(x,y,0)=ax^2+bxy+cy^2=(At^2+bt+C)^2,\: (At^2+bt+C,b)=1.\] Let $x'=\frac{x}{(x,y)},y'=\frac{y}{(x,y)},a'=\frac{At^2+bt+C}{(x,y)} \in R$, then 
	\[ax'^2+bx'y'+cy'^2=a'^2,\: (a',b)=1.\] From the proof of lemma $\ref{1.10}$, there exists $r',s' \in R$ such that $\begin{pmatrix}
		x'&y'\\r'&s'
	\end{pmatrix}
	\in S_{(a,b,c)}$. Then \[ (a,b,c) \approx \begin{pmatrix}
		x'&y'\\r'&s'
	\end{pmatrix} (a,b,c)=(a'^2,b,c'),\text{ where }c'=ar'^2+br's'+cs'^2.\]
	From direct Bhargava cube 
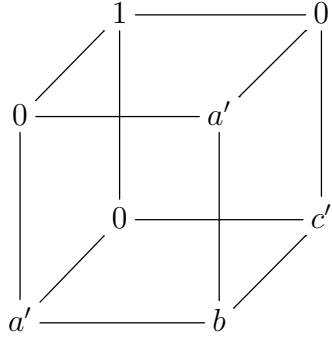
\[\xymatrix{&
		1\ar@{-}[rr]\ar@{-}[ld]\ar@{-}[dd]& &
		0\ar@{-}[ld]\ar@{-}[dd]\\
		0\ar@{-}[rr]\ar@{-}[dd]& &
		a'\ar@{-}[dd]&\\&
		0\ar@{-}[rr]\ar@{-}[ld]& &
		c'\ar@{-}[ld]\\
		a'\ar@{-}[rr]& &
		b&}\]
	we have \[ [(a',b,a'c')][(a',b,a'c')]=[(a'^2,b,c')].\]
	That is, \[[(a,b,c)]=[(a'^2,b,c')]=[(a',b,a'c')]^2.\]
	So $\ker(\omega)=C(b,\Delta)^2$.
\end{proof}

\begin{example}\label{eg2}
	Consider $\ref{eg1}$(2). Since
	\[C(T^2,T^4)=\{[(1,T^2,T^4)],[(T^2+T+1,T^2,T^2+T+1)]\} \cong \mathbf{Z} / 2\mathbf{Z},\]
	we have \[\ker(\omega)=C(T^2,T^4)^2=\{[(1,T^2,T^4)]\}.\] 
	And by computation we have
	\[\omega([(1,T^2,T^4)])=\{[1],[T^2+1],[T^3+1],[T^3+T^2+1]\}=H,\]and
	\begin{equation}\notag\begin{aligned}
		&\omega([(T^2+T+1,T^2,T^2+T+1)])\\=&\{[T+1],[T^2+T+1],[T^3+T+1],[T^3+T^2+T+1]\}=[T+1]H \ne H.
		\end{aligned}\end{equation}
	As a result, if two class $C_1,C_2 \in C(T^2,T^4)$ can represented a same element in $R$ or $(R/(T^4))^*$, then they are the same class.
\end{example}

\begin{example}\label{eg3}
	Consider $\ref{eg1}$(3). Since \[C(1,T^3)=\{[(1,1,T^3)],[(T,1,T^2)],[(T^2,1,T)]\} \cong \mathbf{Z}/3\mathbf{Z}\] we have $C(1,T^3)/C(1,T^3)^2$ is a trivial group. On the other hand, since $b^2=1$ we have $(R/(1))^*/H$ is a trivial group. Moreover, if $b \in \mathbb{F}$, then $C(b,\Delta)$ must be a group of odd order since $(R/(b^2))^*$ is a trivial group and $C(b,\Delta)/C(b,\Delta)^2$ must be a trivial group, that is, there is not an element of order $2$ in $C(b,\Delta)$.
\end{example}

\section{Conclusion}
\qquad This paper extend Gauss's direct composition to quadratic forms over $R$, a polynomial ring of finite field of characteristic 2. Under the generalized direct composition, the set of proper equivalence class of primitive quadratic forms over $R$ is a finite Abelian group and is isomorphic to the corresponding Picard group. Moreover, if two forms of same invariant $b,\Delta \in R$ can represent same elements in $(R/(b^2))^*$, then they differ only by a square term. This extends the classical theory of Gauss  to the case of characteristic 2.

This work does not address the question of reduced form, so it is complex to determine $C(b,\Delta)$. Gauss also considered number of genera( which is equal to number of proper \linebreak equivalence class of order 2). The next step of this project is to compute the number of genera and to find the relation between genus theory in characteristic 2 and class field.
%% ================================================

\addcontentsline{toc}{section}{References}
\begin{center}
	
\end{center}

\begin{center}
	\xiaoerhao\textbf{Acknowledgments}
\end{center}
\addcontentsline{toc}{section}{Acknowledgments}
\qquad I wish to thank my supervisor, Professor Xu, for his thoughtful guidance on the key steps of this paper and for his encouragement. I would also like to thank Professor Fang, in whose class I first encountered the theory of binary quadratic form and developed an interest in it. 

\end{document}